\documentclass[12pt,a4paper]{amsart}
\makeatletter
\renewcommand\normalsize{%
    \@setfontsize\normalsize{11.7}{14pt plus .3pt minus .3pt}%
    \abovedisplayskip 10\p@ \@plus4\p@ \@minus4\p@
    \abovedisplayshortskip 6\p@ \@plus2\p@
    \belowdisplayshortskip 6\p@ \@plus2\p@
    \belowdisplayskip \abovedisplayskip}
\renewcommand\small{%
    \@setfontsize\small{9.5}{12\p@ plus .2\p@ minus .2\p@}%
    \abovedisplayskip 8.5\p@ \@plus4\p@ \@minus1\p@
    \belowdisplayskip \abovedisplayskip
    \abovedisplayshortskip \abovedisplayskip
    \belowdisplayshortskip \abovedisplayskip}
    
\renewcommand\footnotesize{%
    \@setfontsize\footnotesize{8.5}{9.25\p@ plus .1pt minus .1pt}
    \abovedisplayskip 6\p@ \@plus4\p@ \@minus1\p@
    \belowdisplayskip \abovedisplayskip
    \abovedisplayshortskip \abovedisplayskip
    \belowdisplayshortskip \abovedisplayskip}
\ifdefined\pdfpagewidth
\else
\fi
\calclayout
\makeatother

\usepackage{amssymb, amsmath}
\usepackage{amssymb,amsfonts}
\usepackage{amsmath,latexsym}
\usepackage{pdfsync}
\usepackage{bbm}
\usepackage{color,xcolor}
\usepackage{todonotes}
\usepackage{cancel}
\usepackage{diagbox}
\usepackage{enumitem}

\usepackage[latin1]{inputenc}

 \usepackage[all]{xy}

\numberwithin{equation}{section}

\newtheorem{theorem}{Theorem}[section]
\newtheorem{proposition}[theorem]{Proposition}
\newtheorem{lemma}[theorem]{Lemma}
\newtheorem{defi}[theorem]{Definition}
\newtheorem{corollary}[theorem]{Corollary}

\newtheorem{example}[theorem]{Example}

\def\eps{\varepsilon}

\def\curl{\text{curl}\,}

\def\eps{\varepsilon}

\def\bega{\begin{aligned}}
\def\enda{\end{aligned}}
\def\R{\mathbb{R}^2}

\def\R{\mathbb{R}}

\def\bcase{\begin{cases}}
\def\ecase{\end{cases}}

\def\bmx{\begin{bmatrix}}
\def\emx{\end{bmatrix}}

\DeclareMathOperator{\sech}{sech}

\newcommand{\beq}{\begin{equation}}
\newcommand{\eeq}{\end{equation}}
\newcommand{\ben}{\begin{eqnarray}}
\newcommand{\een}{\end{eqnarray}}
\newcommand{\beno}{\begin{eqnarray*}}
\newcommand{\eeno}{\end{eqnarray*}}

\providecommand{\R}{\mathbb{R}}
\providecommand{\N}{\mathbb{N}}

\renewcommand{\div}{\mathrm{div} \,}

\newcommand{\Id}{\operatorname{Id}}
\providecommand{\rot}{\mathrm{rot} \,}

 \renewcommand{\Im}{\operatorname{im}}

\usepackage{cellspace}
  \medskip

\address{  F.~Sueur,  Institut de Math\'ematiques de Bordeaux, UMR CNRS 5251,Universit\'e de Bordeaux, 351~cours
de la Lib\'eration, F33405 Talence Cedex, France}

\date{\today}

\title[Differential transmutations]{ Differential transmutations}

\author{ Franck Sueur}
\begin{document}

\begin{abstract}
Inspired by Gromov's partial differential relations, we introduce a notion of differential transmutation, which allows to transfer some local properties of solutions of a PDE to solutions of another PDE, in particular local solvability, hypoellipticity,  weak and strong unique continuation properties and the Runge property. 
 The latest refers to the possibility to approximate some given local solutions by a global solution, with 
 point force controls 
in preassigned positions in the holes of the space domain. 
 As examples we prove that $2$D Lamé-Navier system and the $3$D steady Stokes system, can be obtained as  differential transmutations of appropriate tensorizations of the Laplace operator. 
\end{abstract}

\maketitle
\setcounter{tocdepth}{3}
\tableofcontents

\section{Introduction}

We introduce a new method to transfer some  properties from a linear operator $B$ to another linear operator $A$.
On the opposite to the classical notion of conjugation, our recipe uses two steps or more.
These steps involve two types of ingredients: 
 the commutativity of some diagrams involving $A$ and $B$ and some B\'ezout-type identities involving  $A$ or $B$.
These identities require a finite 
family of auxiliary operators which
 we call  the  philosopher's stone of the
 transmutation of  $B$ into $A$.
 Actually several kinds of transmutations are considered, involving from $2$ to $6$ equations, while using between $4$ and $6$ auxiliary operators. 
  We first present the main features of this approach in a general setting, in Section \ref{sec-ab}. 
 There, $A$ and $B$ are merely linear mappings, as well as the  auxiliary operators. 
 We investigate in particular the equivalence class properties of some types of transmutations, 
   the degrees of freedom in the philosopher's stone, the stability under the usual 
    algebraic operations including some abstract adjoint operations. 
\smallskip
 
 Then,  in Section \ref{sec-dt}, we investigate the case where  $A$  and $B$ are linear differential operators, with the extra constraint that we  look for auxiliary operators which also are  linear differential operators, all of these operators having smooth variable coefficients. 
  This raises the notion of  differential transmutation for which  the numerical advantage of the number of  auxiliary operators over the number of equations to satisfy  echoes one  motto from Gromov's theory of partial differential relations, see \cite{Gromov}, which, in the linear setting, is that an undetermined  linear differential  operator may have a right-inverse which also is a linear differential operator. 
The gain of looking for auxiliary operators which are all differential, rather than considering some possible integral inverses, is the preservation of local properties. 
Indeed, we prove that appropriate differential transmutations allow to  transfer  local solvability, hypoellipticity, unique continuation properties and Runge's approximation type properties. The latest refers to the possibility to approximate some given local solutions of a partial differential equation by a global solution, with point force controls 
in preassigned positions in the holes of the space domain. 
\smallskip

 As an illustration of our approach we consider in  Section \ref{sec-2D} and in Section \ref{sec-3D}, respectively in $2$D  and in $3$D,  the case of the steady Stokes operator, which is usually thought as a composition of  the Leray-Helmholtz projection, which is a non-local operator, and of the Laplace operator. 
 On the opposite our approach  links the  steady Stokes operator to the Laplace operator by a  differential transmutation, that is by  the means of local, differential, operators. 
 This allows us to recover an earlier result by Fabre and Lebeau on the weak unique continuation of  the  steady Stokes operator and to improve a result by 
  Glass and Horsin on the global approximation with point force controls.
Some of our results actually also encompass the Lamé-Navier operator which appears in elasticity.

\section{Abstract transmutations}
\label{sec-ab}

\subsection{Definitions and  first properties}

For two given vector spaces $\mathcal E$ and $\mathcal F$, we denote $L(\mathcal E ; \mathcal F)$ the space of the  linear mappings from  $\mathcal E$ to $\mathcal F$.
 Let $\mathcal E$, $\mathcal F$, $\mathcal G$ and $\mathcal H$ some non-trivial vector spaces. 
   Let  $A$, $ B$, $P$, $Q$, $R$, $S$, $T$ some linear mappings between some pairs of those spaces according to the diagrams below. 

\begin{align*}
   \xymatrix{
 \mathcal E    \ar@<2pt>[r]^A  \ar@<2pt>[d]_Q& \mathcal F  \ar@<2pt>[d]_S \\
   \mathcal G   \ar@<2pt>[r]^B   &  \mathcal H  
   } 
 & \quad   \quad \quad 
   \xymatrix{
 \mathcal E    \ar@<2pt>[r]^A  & \mathcal F   \\
   \mathcal G   \ar@<2pt>[r]^B  \ar@<2pt>[u]_P  &  \mathcal H  \ar@<2pt>[u]_R
  }  
&\quad    \quad   \quad 
   \xymatrix{
 \mathcal E    \ar@<2pt>[r]^A  \ar@<2pt>[d]_P&  \mathcal F      \ar@<2pt>[l]^T \\   
   \mathcal G     \ar@<2pt>[u]_Q &  
  }  
  &\quad  \quad  \quad   
   \xymatrix{
 \mathcal E    \ar@<2pt>[r]^A  &   \mathcal F     \ar@<2pt>[l]^T   \ar@<2pt>[d]_R \\
 &  \mathcal H  \ar@<2pt>[u]_S }
  \\   \text{Pre-catalysis } &\quad \quad \text{Post-catalysis } &\quad  \quad   \text{Pre-regeneration } &\quad  \text{Post-regeneration }
    \\  BQ  = S A &\quad \quad  AP = R  B &\quad P Q  +TA  = \Id_\mathcal E  &\quad  RS + AT  = \Id_{  \mathcal F}
  \end{align*}

Our first definition below  concerns the commutativity of the two first diagrams above while the condition at  stake in the third and fourth diagrams 
is a B\'ezout-type identity.

 \begin{defi} \label{dispatched} 
 We say
  that  $A$ is a pre-catalysis  of $B$ by $(Q,S)$ if $BQ  = S A$;
  that $A$ is a post-catalysis of $B$ by $(P,R)$ if $AP = R  B$; 
  that $A$ is weakly pre-regenerated by $(P,Q,T)$ if  $\Im (P Q  +TA  - \Id_\mathcal E) \subset \ker A$; 
that $A$ is pre-regenerated by $(P,Q,T)$ if $P Q  +TA  = \Id_\mathcal E$, 
that $A$ is weakly post-regenerated by $(R,S,T)$ if $\Im A \subset \ker  (RS + AT  - \Id_{  \mathcal F})$; 
that $A$ is post-regenerated by $(R,S,T)$ if $RS + AT  = \Id_{  \mathcal F}$.
\end{defi}
The following definition introduces a whole line of   notions of transmutation.
 \begin{defi} 
We say 
that $A$ is a  basic   transmutation of  $B$ by  $(P,Q,R,S)$ if $A$ is a post-catalysis  of $B$ by $(P,R)$ and   a pre-catalysis  of $B$ by $(Q,S)$; 
    that $A$ is a post-transmutation of  $B$ by philosophers' stone $(P,R,S,T)$ if $A$ is a post-catalysis  of $B$ by $(P,R)$   and  $A$ is post-regenerated by $(R,S,T)$; 
      that $A$ is a pre-transmutation of  $B$ by philosophers' stone $(P,Q,S,T)$ if $A$ is a pre-catalysis  of $B$ by $(Q,S)$  and $A$ is pre-regenerated by $(P,Q,T)$; 
 that $A$  is a bronze transmutation of  $B$ by  $(P,Q,R,S,T)$ if $A$ is a  basic transmutation of  $B$ by  $(P,Q,R,S)$ and $A$ is a post-transmutation of  $B$ by philosophers' stone $(P,R,S,T)$;  
 that  $A$ is a silvern  transmutation of $B$ by philosophers' stone $(P,Q,R,S,T)$ if  $A$ is a  basic  transmutation of  $B$ by  $(P,Q,R,S)$ and  $A$ is a pre-transmutation of  $B$ by philosophers' stone $(P,Q,S,T)$; 
and finally that $A$ is a golden transmutation of $B$ by philosophers' stone $(P,Q,R,S,T)$ if it is a bronze transmutation and a silvern transmutation.
\end{defi}
The following table recapitulates our offers. 
$$ 
\begin{tabular}{| c || c |l| c ||}
 \hline			
 \text{Basic}&post-catalysis & $AP = R  B$   \\ 
  &  pre-catalysis &    $BQ  = S A$    \\ \hline  
   \text{Post-transmutation} & post-catalysis  &   $AP = R  B$  \\  
    & post-regeneration &  $RS + AT  = \Id_{  \mathcal F}$ \\  \hline
  \text{Pre-transmutation} & pre-catalysis &  $BQ  = S A$ \\   
    &   pre-regeneration &   $P Q  +TA  = \Id_\mathcal E$  \\  \hline
      \text{Bronze}& post-catalysis  &    $AP = R  B$   \\ 
  &  pre-catalysis &   $BQ  = S A$   \\ 
  &  post-regeneration   &  $RS + AT  = \Id_{  \mathcal F}$  \\ \hline  
  \text{Silvern} & post-catalysis  &   $AP = R  B$  \\  
&  pre-catalysis   &    $BQ  = S A$  \\   
&  pre-regeneration  &   $P Q  +TA  = \Id_\mathcal E$ \\  \hline
     \text{Golden}& post-catalysis   &   $AP = R  B$   \\ 
     & pre-catalysis &   $BQ  = S A$   \\   
     & pre-regeneration  &    $P Q  +TA  = \Id_\mathcal E$  \\ 
     & post-regeneration  &  $RS + AT  = \Id_{  \mathcal F}$ \\ 
\hline  
 \end{tabular}
$$

Next result establishes that silvern transmutations, respectively adjoint transmutations, satisfy the weak post-regeneration property, respectively the weak pre-regeneration property.
 \begin{proposition} \label{prop-post-r}
If $A$ is a  silvern transmutation of  $B$ by philosophers' stone $(P,Q,R,S,T)$ then $A$ is weakly post-regenerated by $(R,S,T)$.  
If  $A$  is an bronze transmutation of  $B$ by  $(P,Q,R,S,T)$ then $A$ is weakly pre-regenerated by $(P,Q,T)$.
 \end{proposition} 
\begin{proof}
To prove the first part, we first observe that 
$(RS + AT)A=RBQ+ A( \Id_\mathcal E - P Q )$ by pre-catalysis  and pre-regeneration. 
Thus, by post-catalysis, we deduce that 
$(RS + AT)A=A$ which provides the result.
To prove the second part, we compute 
$A(P Q  +TA )= RBQ + (\Id_{  \mathcal F} - RS)A$ by post-catalysis and post-regeneration.
Thus, by pre-catalysis
  $A(P Q  +TA )= A $ which concludes the proof.
\end{proof}

Let us also mention here the  following straightforward elementary properties.
 \begin{proposition} \label{prop-elem1}
  If $A$  is a basic  transmutation of $B$ by  $(P,Q,R,S)$, then we have:  $P \ker B \subset \ker A$, $Q \ker A \subset  \ker B$,  $ R \Im B \subset \Im A$,
 and $S \Im A \subset \Im B$.
   If $A$  is a silvern transmutation of $B$ by  $(P,Q,R,S,T)$  then  $\ker A= P \ker B$ and 
  $ \ker A \cap \ker  Q= 0$.
 \end{proposition} 
%

\subsection{Equivalence class properties}

 We start with the following observation on self-transmutation.
 \begin{proposition} \label{prop-elem0}
  Any $B$ in  $L(\mathcal  G; \mathcal H)$ is a golden 
   transmutation of  itself by the philosophers' stone 
 $(\Id_\mathcal G ,\Id_\mathcal G , \Id_\mathcal H,\Id_\mathcal H ,0_\mathcal  G)$.
 \end{proposition} 
  It is clear that if $A$ is a basic  transmutation of $B$ by  $(P,Q,R,S)$
then $B$ is a basic  transmutation of $A$ by  $(Q,P,S,R)$. 
        As a next step towards the reflexivity property of silvern 
    transmutations, 
we investigate the existence of a linear mapping $\tilde T$ such that 
 $QP + \tilde T B = \Id_\mathcal G $, starting with introducing the following  necessary  condition.

 \begin{defi} \label{proper} 
 We say
  that  $A$ is a proper silvern (respectively golden) transmutation of  $B$  by  philosophers' stone $(P,Q,R,S,T)$ if
   $A$ is a silvern (respectively golden) transmutation of  $B$  by  philosophers' stone $(P,Q,R,S,T)$
   and that $ \ker B \cap \ker  P = 0$.
\end{defi}

To prove that the proper condition is also sufficient, we rely on the following classical result, which, in the general case, makes use of the axiom of choice. 
 \begin{lemma} \label{rinverse}
Let $A$ in $  L(\mathcal E ; \mathcal F)$  and $F $ in $  L(\mathcal E ; \mathcal E)$ with $\ker A \subset \ker F$.
Then there exists $T$ in $L(\mathcal F ; \mathcal E)$ such that $TA=F$.  Moreover  $T$ is uniquely determined on $\Im A$. 
\end{lemma}

Next result establishes that proper silvern transmutations can be inverted.
 \begin{proposition} \label{prop-elem2}
  Let $A$  a proper silvern transmutation of $B$ by  $(P,Q,R,S,T)$. 
Then $ \ker B \subset \ker (\Id_\mathcal G -QP) $ and 
there is a linear mapping $\tilde T$, whose restriction to $\Im B$ is uniquely defined, such that  $\tilde T B = \Id_\mathcal G -QP$ 
and such  that  $B$  is a proper silvern  transmutation of $A$ by philosopher's stone $(Q,P,S,R, \tilde T )$. 
 \end{proposition} 
 \begin{proof}
We observe that,  for $v\in \ker B$, $u:=Pv$ is in $\ker A$, by post-catalysis, 
so that, using now the pre-catalysis,  $BQPv = SAPv=0$  and, by the pre-regeneration, $u = PQ u$. 
Hence,  $v-QPv \in  \ker B \cap \ker P$,  which leads, by the compatibility condition, to $v$ in $\ker (\Id_\mathcal G -QP) $. 
Thus  $ \ker B \subset \ker (\Id_\mathcal G -QP) $, and
then, by  Lemma \ref{rinverse}, 
there is  a linear mapping $\tilde T$, whose restriction to $\Im B$ is uniquely defined, such that  $\tilde T B = \Id_\mathcal G -QP$.
Clearly,  $B$  is a basic transmutation of $A$ by  $(Q,P,S,R)$, and therefore a  silvern  transmutation of $A$ by philosopher's stone $(Q,P,S,R, \tilde T )$, with the previous result.  
  Moreover, by Proposition \ref{prop-elem1},  $ \ker A \cap \ker  Q=0$,  
    so that     $B$  is a proper silvern   transmutation of $A$ by  $(Q,P,S,R,\tilde T  )$. 
\end{proof}
 Now, to study the transitivity of the notion,  let $\mathcal E$, $\mathcal F$, $\mathcal G$, $\mathcal H$, $\tilde{\mathcal G}$ and $\tilde{\mathcal H}$ some vector spaces. 
   Let  $A, B,P_1,Q_1 ,R_1 ,S_1 ,T_1 ,{P}_2 ,{Q}_2 ,{R}_2 ,{S}_2 ,{T}_2 $ some linear mappings between some pairs of those spaces according to the following diagrams:
\begin{gather*}
 \xymatrix{ 
& \mathcal E    \ar@<2pt>[r]^A  \ar@<2pt>[d]_{P_1} & \mathcal F  \ar@<2pt>[d]_{R_1}  \\
\text{Catalysis: }&   \mathcal G   \ar@<2pt>[r]^B  \ar@<2pt>[u]_{Q_1}   \ar@<2pt>[d]_{P_2}  &  \mathcal H  \ar@<2pt>[u]_{S_1}  \ar@<2pt>[d]_{R_2}\\
 & \tilde{\mathcal G}   \ar@<2pt>[r]^C  \ar@<2pt>[u]_{Q_2}  & \tilde{\mathcal H}  \ar@<2pt>[u]_{S_2}
  }  
 \quad    \quad   \quad  \quad    \quad \quad 
   \xymatrix{
& \mathcal E    \ar@<2pt>[r]^A  \ar@<2pt>[d]_{P_1} &    \mathcal F    \ar@<2pt>[d]_{S_1}  \ar@<2pt>[l]^{T_1} \\   
 \text{Pre-regeneration: }&  \mathcal G   \ar@<2pt>[r]^B  \ar@<2pt>[u]_{Q_1} \ar@<2pt>[d]_{P_2}&  \mathcal H      \ar@<2pt>[l]^{T_2} \\  
&   \tilde{\mathcal G}       \ar@<2pt>[u]_{Q_2}&
  }  
  \end{gather*}

 \begin{proposition} \label{prop-elem3}
    If $A$  is a  silvern  transmutation of $B$ by  $(P_1,Q_1,R_1,S_1,T_1)$
 and $B$ is a  silvern  transmutation of $C$ by  $({ P}_2,{ Q }_2,{ R }_2,{ S }_2,{ T }_2)$, 
 then   $A$ is   a  silvern  transmutation of  $C$ by  $(P_1 P_2,Q_2 Q_1,R_1  R_2, S_2 S_1,T_1+P_1  T_2 S_1)$.
Moreover, if the two first transmutations are  proper silvern  transmutations then so is the third. 
 \end{proposition} 
 \begin{proof}
The first part is   a direct computation. For the last part, we observe that, 
by  Proposition \ref{prop-elem2}, there are some  linear mappings $\tilde T_1$ and $\tilde T_2$ such that 
$B$ (respectively $C$) is a proper silvern  transmutation of $A$ (resp. $B$)  by  
$$(Q_1,P_1,S_1,R_1, \tilde T_1)  \quad   \text{(resp. }  ( Q_2, P_2, S_2, R_2,  \tilde T_2  )).$$
Thus, by the first part, $C$  is  a proper silvern    transmutation of $A$ by  
 $$( Q_2 Q_1,P_1  P_2, S_2 S_1,R_1 R_2,  \tilde T_2 +  Q_2 \tilde T_1   R_2) .$$
 Then, by the last property of Proposition \ref{prop-elem1}, we infer that $ \ker C \cap \ker  P_1  P_2= 0$, so that   $A$ is   a proper silvern  transmutation of  $C$ by  $(P_1 P_2,Q_2 Q_1,R_1  R_2, S_2 S_1,T_1+P_1  T_2 S_1)$.
\end{proof}

Observe that, as a consequence of the previous properties,  proper silvern  transmutations  share the properties  of an equivalence relation. 
The following characterisation determines whether or not two linear mappings are equivalent in terms of the existence of a linear bijective mapping between the kernels.
The associated philosopher's stones benefit from some degrees of freedom. 
\begin{theorem} \label{th:classes}
  If $A$  is a proper silvern transmutation of $B$ by  $(P,Q,R,S,T)$ then the restriction of $P$ to $\ker B$ is a linear bijection from  $\ker B$ onto  $\ker A$ and 
  the restriction of $Q$ to $\ker A$ is a linear bijection from  $\ker A$ onto  $\ker B$. Moreover  the restriction of $PQ$ to $\ker A$ is the identity. 
  Conversely, if there is a linear bijection $P$ from $\ker B$ onto  $\ker A$, 
   considering its inverse  $Q$ from  $\ker A$ onto  $\ker B$, then 
   for any linear extension of $P$  to $\mathcal G$, 
   for any linear extension of $Q$ to $\mathcal E$, 
     for any linear  triple of mappings  $R$,  $S$ and $T$   in the non-trivial linear space determined by the equations: $RB=AP$, $SA=BQ$ and $TA = \Id - PQ$, 
   we have that  $A$  is a proper silvern transmutation of $B$ by  $(P,Q,R,S,T)$. 
\end{theorem}
 \begin{proof}
  Let us assume that $A$  is a proper silvern transmutation of $B$ by philosopher's stone $(P,Q,R,S,T)$.
According to  second to last property of Proposition \ref{prop-elem1},  
and since  $ \ker B \cap \ker  P = 0$,
  the restriction of $P$ to $\ker B$ is a linear bijection from  $\ker B$ onto  $\ker A$.
By  Proposition \ref{prop-elem2}, we similarly have that  the restriction of $Q$ to $\ker A$ is a linear bijection from  $\ker A$ onto  $\ker B$.
By pre-regeneration,  the restriction of $PQ$ to $\ker A$ is the identity. 
 Conversely, assume that  there is a linear bijection $Q$ from  $\ker A$ onto  $\ker B$. Let us denote  $P$ its inverse from $\ker B$ onto  $\ker A$ and consider 
 any linear extension of $P$  to $\mathcal G$, 
   any linear extension of $Q$ to $\mathcal E$. 
   Since  $\ker B \subset \ker AP$, $\ker A \subset \ker BQ$ and $\ker A \subset \ker ( \Id -PQ)$, by 
    Lemma \ref{rinverse}, there are some  linear mappings  $R$,  $S$ and $T$   
    such that $AP=RB$, $BQ=SA$ and $TA = \Id - PQ$,  and they are uniquely defined, respectively, on 
    $\Im B$, $\Im A$ and $\Im A$.   
 Since  $ \ker B \cap \ker  P = 0$,  we conclude that $A$  is a proper silvern transmutation of $B$ by  $(P,Q,R,S,T)$. 
    \end{proof}
Actually Theorem \ref{th:classes} above connects the transmutation viewpoint introduced here with the 
issue of the classification of the spaces of solutions  of  linear systems through $D$-modules theory, considered in  algebraic analysis since Malgrange, see \cite{MalgrangeB}, and   \cite{CQ} for a more recent account on the subject. 
\medskip

Let us mention the following counterpart of Theorem \ref{th:classes}  for silvern transmutations.
\begin{theorem} \label{th:classes-st}
  If $A$  is a silvern transmutation of $B$ by  $(P,Q,R,S,T)$ then 
    the restriction of $Q$ to $\ker A$ is a linear injective mapping from  $\ker A$ into  $\ker B$, 
  $ P \ker B  \subset \ker A $, 
  and  the restriction of $PQ$ to $\ker A$ is the identity. 
  Conversely, if there is a linear injective mapping $Q$ from  $\ker A$ into  $\ker B$, considering  any linear extension of  
   its inverse $P$  from  $Q\ker A$ onto  $\ker B$, and any linear extension of $Q$  to $\mathcal E$, then 
   $A$  is a silvern  transmutation of $B$ by  $(P,Q,R,S,T)$ where $(R,S,T)$ is  any  triple of mappings  $R$,  $S$ and $T$   in the non-trivial linear space determined by the equations: $RB=AP$, $SA=BQ$ and $TA = \Id - PQ$. 
\end{theorem}

Let us now turn to the case of golden transmutations which we reinforce as follows.

 \begin{defi} \label{trans} 
We say that $A$ is a two-sided golden transmutation of $B$ by philosophers' stone $(P,Q,R,S,T,\tilde T)$ if 
  $AP = R  B$, $BQ  = S A$, $P Q  +TA  = \Id_\mathcal E$,  $RS + AT  = \Id_{  \mathcal F}$,   $QP  + \tilde T B  = \Id_\mathcal G$ and  $SR +  B \tilde T  = \Id_{  \mathcal H}$. 
\end{defi}
$$ 
\begin{tabular}{| c || c |l| c ||}
 \hline			
     \text{Two-sided  golden}& post-catalysis   &   $AP = R  B$   \\ 
     & pre-catalysis &   $BQ  = S A$   \\   
     & pre-regeneration  &    $P Q  +TA  = \Id_\mathcal E$  \\ 
     & post-regeneration  &  $RS + AT  = \Id_{  \mathcal F}$ \\ 
     & backward pre-regeneration  &    $QP  + \tilde T B  = \Id_\mathcal G$  \\ 
     & backward post-regeneration  &  $SR +  B \tilde T  = \Id_{  \mathcal H}$ \\ 
\hline  
 \end{tabular}
$$
\begin{align*}
   \xymatrix{
 \mathcal E   \ar@<2pt>[d]_P&   \\
   \mathcal G    \ar@<2pt>[u]_Q \ar@<2pt>[r]^B   &  \mathcal H   \ar@<2pt>[l]^{\tilde T} 
   } 
  &\quad  \quad  \quad   \quad 
   \xymatrix{
& \mathcal F   \ar@<2pt>[d]_R \\
   \mathcal G    \ar@<2pt>[r]^B   &  \mathcal H   \ar@<2pt>[l]^{\tilde T}  \ar@<2pt>[u]_S
 }
  \\   \text{Backward pre-regeneration } &\quad  \quad \quad \text{Backward post-regeneration }
    \\    QP  + \tilde T B  = \Id_\mathcal G  &\quad \quad \quad  SR +  B \tilde T   = \Id_{  \mathcal H}
  \end{align*}
\begin{theorem} \label{th:classes-2gold}
Two-sided  golden transmutations share the properties of a equivalence relation. More precisely, we have the following. 
Any $A$  is a  two-sided golden transmutation of itself by philosophers' stone $(\Id,\Id,\Id,\Id,0,0 )$. 
  If $A$  is a  two-sided golden transmutation of $B$ by  $(P,Q,R,S,T,\tilde T )$ then 
   $B$  is a  two-sided golden transmutation of $A$ by  $(Q,P,S,R,\tilde T,T)$.
     If $A$  is a  two-sided golden transmutation of $B$ by  $(P_1,Q_1,R_1,S_1,T_1, \tilde T_1 )$
 and $B$ is a  two-sided golden transmutation of $C$ by  $({ P}_2,{ Q }_2,{ R }_2,{ S }_2,{ T }_2,\tilde T_2)$, 
 then   $A$ is   a  two-sided golden transmutation of  $C$ by  $\big(P_1 P_2,Q_2 Q_1,R_1  R_2, S_2 S_1,T_1+P_1  T_2 S_1, \tilde T_2 + Q_2 \tilde T_1 R_2 \big)$.
\end{theorem}
%
 
%
\subsection{Some examples}
The following straightforward result states that a conjugation is a two-sided golden transmutation.
 \begin{proposition} \label{prop-conj}
 If $A = PBP^{-1}$ then, for any linear mapping $T$, 
 $A$ is a two-sided  golden 
  transmutation of $B$ by  $(P,Q ,R,S,T,\tilde T)$, where
 $Q := P^{-1} - P^{-1} TA$, $ R :=P$,  $ S:= P^{-1} - B P^{-1} T$ and
 $\tilde T := P^{-1} T P$.
 \end{proposition} 

\medskip

Next example concerns the case where the transmutation aims at permuting the two factors of an operator. 
 \begin{proposition} \label{prop-perm}
 Let $\mathcal E$ and $\mathcal G$ some vector spaces. 
   Let  $M$ in  $L(\mathcal  E; \mathcal G)$
   and $L$ in  $L(\mathcal  G; \mathcal E)$. 
 Set $A:= LM$ and $B:=ML$.
  Then     $A$ is a  basic   transmutation of  $B$ by  $(L,M,L,M)$.
  Moreover  there is $T$ such that $A$ is a silvern transmutation of  $B$  by  $(L,M,L,M,T )$
  if and only $\ker  A = 0  $, and in this case  $T$ is uniquely determined on $\Im A$ by 
 $T A=  \Id - A$,  and
   $A = (\Id + T)LBM$. 
  Finally, in this case, 
   this  transmutation  is proper  if and only $\ker  L =0  $. 
 \end{proposition} 

 Before to proceed to the proof, let us illustrate the result with the following diagrams.
\begin{align*}
&\xymatrix{
 \mathcal E    \ar@<2pt>[r]^M  \ar@<2pt>[d]_L & \mathcal G  \ar@<2pt>[r]^L  & \mathcal E  \ar@<2pt>[d]_L   \\
   \mathcal G   \ar@<2pt>[r]^L  \ar@<2pt>[u]_M  &  \mathcal E  \ar@<2pt>[r]^M  &  \mathcal G  \ar@<2pt>[u]_M 
  }  &\quad \quad\quad\quad\quad\quad
   \xymatrix{
 \mathcal E    \ar@<2pt>[r]^A  \ar@<2pt>[d]_L&  \Im A      \ar@<2pt>[l]^T \\   
   \mathcal G     \ar@<2pt>[u]_M &  
  }  
   \\ & \quad \quad \text{Catalysis } &\quad \quad\quad\quad\quad\quad  \text{Pre-regeneration }
\end{align*}
 \begin{proof}
First, $A L = LML = LB$ and $BM =  MLM = MA$ so that  $A$ is a  basic   transmutation of  $B$ by  $(L,M,L,M)$.
  Next, by Lemma \ref{rinverse}, there is $T$   such that $TA= \Id - A$ if and only $\ker  A \subset \ker ( \Id - A)$ that is 
   if and only $\ker  A = 0  $, and in this case  $T$ is uniquely determined on $\Im A$ by 
 $T A=  \Id - A$. 
   Moreover, still in this case, 
   $ (\Id + T)LBM =  (\Id + T)A^2 =A$, 
so that  $A$ is the only silvern  transmutation of $B$ by  philosopher's stone $( L,M,L,M,T )$. 
  Finally, in this case,  this  transmutation  is proper  by definition if and only $\ker ML \cap \ker  L =0  $, that is    if and only $\ker  L =0  $. 
\end{proof}
%

\subsection{Degrees of freedom}
\label{sub-free}

  The  following straightforward result emphazises the 
   degrees of freedom in the philosopher's stone  
  in the case where $A$ and $B$ are  invertible.

 \begin{proposition} \label{prop-invert}
 Assume that  $A$  and  $B$ are invertible.
 Then for any linear mappings $R$ and $S$, 
   $A$  is a  two-sided  golden  transmutation of $B$ by  $(P,Q,R,S,T,\tilde T)$ 
 with   $P:= A^{-1}  RB$, $Q:= B^{-1}  SA$, 
 $T = A^{-1}   - A^{-1}  R S $ and $\tilde T := B^{-1}   - B^{-1}  SR $.
 \end{proposition} 

  Let us also mention the following observations on dilatations of philosopher's stone
  in the general case. We set, for $\mu$ and $\lambda$  in $\R$, 
  $$ D_{\mu,\lambda} : (P,Q,R,S) \mapsto (\mu P, \lambda  Q,\mu R,\lambda S) .$$
 \begin{proposition} \label{prop-dilatation}
 Assume that $A$ is  a basic transmutation of $B$ by  $(P,Q,R,S)$. 
 Then for any $\lambda$ and $\mu$ in $\R$,  $A$ is  a basic transmutation of $B$ by  $D_{\mu,\lambda}  (P,Q,R,S)$.
 Assume that $A$ is  a silvern (respectively bronze, golden)  transmutation of $B$ by  philosopher's stone $(P,Q,R,S,T)$.
 Then for any  $\lambda$ in  $\R^*$, setting $\mu := \lambda^{-1}$,  
  $A$ is  a silvern (respectively bronze, golden)  transmutation of $B$ by  $(D_{\mu,\lambda}  (P,Q,R,S),T)$.
 \end{proposition} 

 The following straightforward result emphasizes another possible degree of freedom which may be thought as a gauge invariance in the philosopher's stone.
 As all along the paper it is understood that the operators at stake are defined on the appropriate spaces for the different compositions  to make sense.
 \begin{proposition} \label{prop-gauge}
 Assume that $A$ is  a silvern (respectively bronze) transmutation of $B$ by  $(P,Q,R,S,T)$.
 Let $(\hat Q,\hat S,\hat T)$ such that 
 \begin{equation}
  \label{gequa}
B\hat  Q  = \hat  S A \quad   \text{ and  } \quad  P \hat Q  = -\hat  TA  \, \text{ (resp.  }  R \hat  S = - A  \hat  T).
\end{equation}
Then,  $A$ is  a silvern (resp. bronze) transmutation of $B$ by  $(P,Q+\hat  Q ,R,S+ \hat S,T+\hat T)$.
 Conversely, if $A$ is  a silvern (resp. bronze) transmutation of $B$ by  $(P,Q,R,S,T)$ and also by $(P,Q+\hat  Q ,R,S+ \hat S,T+\hat T)$, then   \eqref{gequa} holds true.
  \end{proposition} 
 Observe that the two gauge equations in   \eqref{gequa} are linear. 
 In the case of two-sided  golden  transmutations five gauge equations are required. 
 \begin{proposition} \label{prop-gauge-2-sided}
 Assume that $A$ is  a two-sided  golden  transmutation of $B$ by  $(P,Q,R,S,T, \tilde T)$.
 Let $(\hat Q,\hat S,\hat T, \check T)$ such that 
 \begin{equation}
  \label{gequagold}
B\hat  Q  = \hat  S A , \,   P \hat Q  = -\hat  TA , \,    R \hat  S = - A  \hat  T   , \,    \hat Q P = -   \check  T  B\,    \text{ and  } \,    \hat  S R = - B  \check  T  .
\end{equation}
Then,  $A$ is  a two-sided  golden  transmutation of $B$ by  $(P,Q+\hat  Q ,R,S+ \hat S,T+\hat T, \tilde T + \check T)$.
 Conversely, if $A$ is  a two-sided  golden transmutation of $B$ by  $(P,Q,R,S,T, \tilde T)$ and also by $(P,Q+\hat  Q ,R,S+ \hat S,T+\hat T,\tilde T +   \check  T )$, then   \eqref{gequagold} holds true.
 \end{proposition} 
 A similar result 
 holds for a shift of $(P,R,T)$, rather than $(Q,S,T)$. 
 \begin{proposition} \label{prop-gauge-2}
 Assume that $A$ is  a silvern  (respectively bronze) transmutation of $B$ by  $(P,Q,R,S,T)$.
 Let $(\hat P,\hat R,\hat T)$ such that 
 \begin{equation}
  \label{gequa2}
A\hat  P  = \hat  R  B \quad   \text{ and  } \quad  \hat  P Q  = -\hat  TA  \, \text{ (Resp.  }  \hat R  S = - A  \hat  T).
\end{equation}
Then,  $A$ is  a silvern  (resp. bronze) transmutation of $B$ by  $(P+\hat  P,Q ,R+ \hat R,S,T+\hat T)$.
 Conversely, if $A$ is  a silvern  (resp. bronze) transmutation of $B$ by  $(P,Q,R,S,T)$ and also by $(P+\hat  P,Q ,R+ \hat R,S,T+\hat T)$, then   \eqref{gequa2} holds true. 
  \end{proposition} 
 %

\subsection{Algebraic operations}

Composing operators in transmutations requires many compatibilities, as illustrated by the following result.
 \begin{proposition} \label{prop-elem4}
    If for $i=1,2$, 
    $A_i$  is a silvern  transmutation of $B_i$ by philosophers' stone  $(P_i,Q_i,R_i,S_i,T_i)$ 
and if there are some linear mappings $ \hat R_2$,  $ \check R_2$,  $ \hat S_2$,  $ \check S_2$ such that 
$P_1 \hat R_2  = R_2$ with $B_1 \hat R_2 =\check   R_2 B_1$, 
$Q_1 \hat S_2  = S_2$, with $A_1 \hat S_2 =\check   S_2 A_1$ and $\ker A_1 \subset \ker T_2$
 then   there exists $T$, uniquely determined on $\Im A_1$, 
  such that 
 $A_1 A_2$ is   a silvern  transmutation of  $B_1 B_2$ by  $(P_2,Q_2,R_1  \check R_2,S_1  \check S_2,T)$.
Moreover, if the second transmutation (for $i=2$) is proper then so is the compounded one. 
 \end{proposition} 
 \begin{proof}
First
 $A_1 A_2 P_2= A_1  R_2 B_2 = A_1 P_1 \hat R_2  B_2 = R_1 \check   R_2 B_1 B_2$ and
 $B_1 B_2 Q_2 = B_1 S_2 A_2 =  B_1 Q_1 \hat S_2  A_2 = S_1 A_1 \hat S_2  A_2 =S_1  \check  S_2  A_1 A_2 $. 
 Moreover, by Lemma \ref{rinverse},
 there exists $T$, uniquely determined on $\Im A_1$, 
  such that $TA_1= T_2$ so that 
 $P_2 Q_2 + T A_1 A_2 = P_2 Q_2 + T_2 A_2 = \Id$. 
 Thus   $A_1 A_2$ is   a silvern  transmutation of  $B_1 B_2$ by  $(P_2,Q_2,R_1  \hat R_2,S_1  \hat S_2,T )$.
 The last claim follows from $\ker B_2 \subset \ker B_1 B_2$.
 \end{proof}

The general situation for addition of operators is even more demanding in terms of compatibility. 
At least, two  silvern  transmutations with 
the  same philosophers' stone can be added.
 \begin{proposition} \label{prop-elem5}
    If for $i=1,2$, 
    $A_i$  is a silvern  transmutation of $B_i$ by the same philosophers' stone  $(P,Q,R,S,T)$ 
 then   $A_1 + A_2$ is   a silvern  transmutation of  $B_1 + B_2$ by  $(P,Q,R,S,T/2) $.
 \end{proposition} 

On the other hand, transmutations work fine with tensorizations. 
 \begin{proposition} \label{prop-elem6}
    If for any $n$ in $\N^*$, for $1 \leq i \leq n$, $A_i$ is a silvern  (respectively bronze, silvern) transmutation of $B_i$ by philosophers' stone  $(P_i,Q_i,R_i,S_i,T_i)$ 
 then   $A_1 \otimes \ldots \otimes A_n$ is   a silvern (resp. bronze, silvern) transmutation of  $B_1  \otimes \ldots \otimes B_n$ by  $(P,Q,R,S,T) $ with
 $U :=    U_1  \otimes \ldots \otimes U_n$ 
 where  $U$ stands for $P,Q,R,S$, and 
$T$.
 \end{proposition} 

One may also increase the degree of self-tensorization by a silvern transmutation.
 \begin{proposition} \label{prop-elem7}
   For any $p$ and $q$ in $\N^*$, with $p\leq q$, 
   for any linear mapping $B$, 
     the tensorization  $B^{\otimes p}$ of $B$, $p$-times by itself, 
     is a  silvern transmutation of $B^{\otimes q} $ by philosophers' stone  $(P,Q,P,Q,0)$ where     
     $P$ maps $x=(x_1,\ldots,x_q) $ in $\mathcal G^q$ to 
     $x=(x_1,\ldots,x_p) $ in $\mathcal G^p$ and
     $Q$ maps  $x=(x_1,\ldots,x_p) $ in $\mathcal G^p$  to $x=(x_1,\ldots,x_p,0,\ldots,0) $ in $\mathcal G^q$.
     If $\ker B = 0$ then this transmutation is golden. Otherwise 
     this transmutation is golden if and only if $p=q$. 
      \end{proposition} 
%
\subsection{Abstract adjoints}

We assume that there is an operation denoted by $*$ which satisfies the following properties for any linear mapping $A$:
 $(A^*)^* = A$, $(AB)^* = B^* A^*$, $(A+B)^* =  A^* + B^*$ and $\Id^* = \Id$. Such properties are satisfied by adjoint operators in Hilbert theory, as well as by the formal adjoint operators of differential operators, as considered in the next section. Such an operation allows to link any silvern transmutation to a bronze transmutation of their adjoints. 
 \begin{proposition} \label{prop-elem8}
 $A$ is  a pre-transmutation of $B$ by  $(P,Q ,S,T)$ if and only if 
   $A^*$ is  a post-transmutation of $B^*$ by  $(S^*,Q^*,P^*, T^*)$. 
  Similarly  $A$ is  a silvern transmutation of $B$ by  $(P,Q ,R,S,T)$ if and only if 
   $A^*$ is  a bronze transmutation of $B^*$ by  $(S^*,R^*,Q^*,P^*, T^*)$.
        \end{proposition} 

The following elementary observation allows to obtain a basic transmutation from a single catalysis. in the case where the formal adjoints $A^*$ and  $B^*$ are respectively left and right multiplications of $A$ and $B$ by some linear mappings.  
 \begin{proposition} \label{prop-trans}
 If $A$ and $B$ are  linear mappings such that there are two  linear  mappings $S_0$ and $Q_0$ satisfying 
 $A^* = S_0 A$ and $B^* = B Q_0$, and such 
that  $A$ is a post-catalysis  of $B$ by $(P,R)$ 
 then $A$ is a pre-catalysis  of $B$ by $(Q,S)$ with 
 $Q= Q_0 R^*$ and $S=P^*  S_0$. 
 \end{proposition} 
 Similarly, if one benefits from some more symmetry properties of the auxiliary operators 
 one may deduce that is a two-sided golden transmutation with less effort.
 \begin{proposition} \label{prop-t}
 If $AP=RB$, $PQ+TA=\Id$, $QP+\tilde T B = \Id$, 
 $P^* = \lambda S$, $Q^* = \lambda^{-1} R$, $A^*= A$, $B^* = B$, $T^* =T$, $\tilde T^* =  \tilde T$, with 
 $\lambda$ in $\R^*$ then $A$ is a two-sided golden transmutation of $B$ by $(P,Q,R,S,T,\tilde T)$.
 \end{proposition} 
 %

\section{Differential transmutations}
\label{sec-dt}

To transfer some local properties of solutions of partial differential equations, 
we introduce the following notion of differential transmutation of  linear differential operators. Here, and in what follows, by linear differential operators we mean  linear differential operators  with $C^\infty$ matrix-valued coefficients. 

\subsection{Definitions and first properties}
 \begin{defi} \label{distri} 
 Let  $d,n,m,p,q$ in $\N^*$.
 Let $A$ and $B$ some differential operators with coefficients respectively in 
   $C^\infty (\R^d ; \R^{m \times n} )$ and  $C^\infty (\R^d ; \R^{q \times p}  )$.
   We say that $A$ is a  silvern (respectively basic, bronze, golden) differential    transmutation of  $B$ by philosophers' stone $(P,Q,R,S,T)$ if 
    $A$ is a silvern (resp. basic, bronze, golden) transmutation of  $B$ by philosophers' stone $(P,Q,R,S,T)$ and 
   if   $P,Q,R,S,T$ are also some linear differential operators with  $C^\infty$ coefficients. 
\end{defi}
Above it is understood that the spaces which are involved in Definition \ref{distri} are 
 $\mathcal E =C^\infty (\R^d ;  \R^n )$ and $\mathcal F= C^\infty (\R^d ; \R^m )$, 
  $\mathcal G= C^\infty (\R^d ; \R^p )$ and $\mathcal H=C^\infty (\R^d ; \R^q)$. 
  Let us observe  that the identities involved in 
Definition  \ref{dispatched} can then be extended to  distributions by density. 

\medskip
By Proposition \ref{prop-elem0} and Proposition \ref{prop-elem3}
we have the following result. 

 \begin{proposition} \label{prop-elem-diff}
 Any  linear differential operator $B$  is the golden differential  transmutation of  itself by philosophers' stone 
 $(1 ,1 , 1,1 ,0)$. 
    Moreover, if $A$  is a silvern (respectively bronze, golden) differential transmutation of $B$ by  $(P,Q,R,S,T)$
 and $B$ is a silvern (resp. bronze, golden) differential transmutation of $C$ by  $({\tilde P},{\tilde Q},{\tilde R},{\tilde S},{\tilde T})$, 
 then   $A$ is   a silvern  (resp. bronze, golden) differential transmutation of  $C$ by  $({P\tilde P},{\tilde Q Q},{R \tilde R},{\tilde S S},T+P {\tilde T} S )$.
 \end{proposition} 

However, unlike the result  in Proposition \ref{prop-elem2} in the abstract setting, one cannot conclude in general that proper silvern differential 
 transmutations  are symmetric nor transitive. 
 While Theorem \ref{th:classes} does not hold anymore too for differential transmutations, we introduce the two following notions  for linear differential operators, which also have the properties of  equivalence relations.
 \begin{defi} \label{DiffEqui} 
 We say that $A$ is  a  two-sided silvern  differential  transmutation of 
 $B$ by philosopher's stone $(P,Q,R ,S,T,\tilde T)$ if 
  $AP = R  B$,   $BQ  = S A$, $P Q  +TA  = \Id$ and 
 $ QP  + \tilde T B  = \Id_\mathcal E$. 
 We say that $A$ is  a  two-sided golden  differential  transmutation of  
 $B$ by philosopher's stone $(P,Q,R ,S,T,\tilde T)$ if 
  $AP = R  B$,   $BQ  = S A$, $P Q  +TA  = \Id_\mathcal E$,  $RS + AT  = \Id$, 
 $ QP  + \tilde T B  = \Id$ and $SR + B\tilde T  = \Id$. 
\end{defi}

Despite the differential constraint lowers the possibilities of philosopher's stones, compared to the first section, we still benefit from a numerical advantage, with a critical ratio of one in the last case  for which six auxiliary operators are to be found to satisfy six equations.
\medskip

In the following, 
we will make use of formal adjoints.
 \begin{defi} \label{formally-adjoint} 
 For a  linear differential operator
  $$A = \sum_{| \alpha | \leq m} \, C_{\alpha} (x) \, \partial^{\alpha} ,$$
  where $\alpha$ is in $\N^d$, $| \alpha |$ denotes its length and the $C_{\alpha}$ are smooth fields of matrices, 
we denote 
$$A^* = \sum_{| \alpha | \leq m} \, (-1)^{| \alpha | }   \partial^{\alpha} \, C_{\alpha} (x) ,$$
its formal adjoint.
\end{defi}
It follows from repeated integrations by parts that the formal adjoint $A^*$ satisfies for any $u$ and $v$ in $C^\infty_c (\R^d)$, 
 \begin{equation}
  \label{sym}
  <Au,v>_{L^2} =  < u,A^* v>_{L^2} ,
\end{equation}
where the brackets stand for the scalar product in the Lebesgue space $L^2  (\R^d)$, 
  and this property uniquely determines   the formal adjoint. 
\medskip

\subsection{Examples}
Below we give a first few examples to illustrate this approach. 

\begin{example} \label{Matsu}
Let $A$ and $P$ two linear differential operators, both scalar valued and with constant complex coefficients. 
Suppose that their characteristic varieties are disjoints. Then 
it follows from Hilbert's Nullstellensatz that there are 
two scalar valued  constant coefficients linear differential operators $Q$ and $T$ such that  $PQ+TA=1$.
Hence $A$  is a golden  differential transmutation of $PA$ 
by  $(P,Q,1,PQ,T)$. 
 \end{example}

Next example shows that the condition of the previous example is not necessary. 
\begin{example} \label{Matsuplus}
Let $A:= \partial_x $ and $B := \partial_x^2$. Then 
$A$  is a golden  differential transmutation of $B$ 
by  $(  \partial_x,1+x,1,(1+x) \partial_x + 2 ,-(1+x))$.
On the other hand there are no linear differential operators with polynomial coefficients  $(P,Q,R,S,T)$ such that 
$B$    is a silvern  differential  transmutation of $A$ by philosophers' stone $(P,Q,R,S,T)$.
 \end{example}

The following example  appears at least  in two different context: in the linearization of the viscous Burgers equations at a steady (or travelling wave) solution 
with a hyperbolic tangent profile, see \cite{Kr2}; and 
in  ferromagnetism, more particularly from the linear stability of a particular solution of the Landau-Lifshitz-Gilbert equations modelling a domain wall, see  \cite{Carbou}.
\begin{example} \label{excar}
Let us consider the following pair of $1$D Schr\"odinger operators: 
\begin{equation} \label{double}
A:= -\partial_x^2 + 2 \tanh^2 (x) - 1   \, \text{ and  } \, B:= -\partial_x^2 + 1 .
\end{equation}  
We introduce the  differential operator $M:= \partial_x + \tanh x$ and observe that $A = M^* M $ and $B= M M^*$, that 
 the kernel of $A$ is one-dimensional, spanned by  the $C^\infty$ function $v$ defined by $v(x) := \sech (x)$, while the one of $B$ is trivial. 
Therefore, by Proposition \ref{prop-perm}, on the one hand, 
$A$ is a basic transmutation of $B$ by $(M^*,M,M^*,M)$, but there is no $T$ for $A$ to be a silvern, nor a bronze, transmutation of $B$ by philosopher's stone  $(M^*,M,M^*,M ,\tilde T)$, while, 
 on the other hand  $B$ is a golden transmutation of $A$ by $(M,M^*,M,M^*,\tilde T)$,  where 
 $\tilde T = B^{-1} - \Id$, where $B^{-1}$ is the convolution operator by the kernel function $\frac12 e^{-|x|}$.
 In particular, this transmutation is not differential.
  Let us also observe that on the one hand,  if a distribution $v$ satisfies $Bv = \delta_0$ in a neighbourhood of $0$ then 
the distribution $u:= M^* v$ satisfies $Au = M^*  \delta_0  =  -  \delta'_0 $; and, 
 on the other hand, if a distribution $u$ satisfies $Au = \delta_0$ in a neighbourhood of $0$ then 
the distribution $v:= Mu$ satisfies $Bv = M \delta_0  =   \delta'_0 $.
 Therefore, despite homogeneous solutions are transferred by the transmutation, see the 
 second to last property of Proposition \ref{prop-elem1}, 
   fundamental solutions are not. 
\end{example}
 Obviously, conjugation by non-vanishing $C^\infty$ functions provide some bronze differential transmutations. This is illustrated in the following example. 
\begin{example} \label{excar2}
Let $V$ a function in  $C^\infty (\R^d ; \R)$  and $A = - \Delta + V$.
Assume that  there are  $\mu$ in $\R$ and $u_0$ in  $C^\infty (\R^d ; \R^*)$ satisfying $A u_0 = \mu $. Set
$$B:= - u_0^{-2} \div (u_0^{2} \nabla \cdot) + \mu u_0^{-1} .$$
 Then $A$ is   a two-sided golden differential transmutation of 
 $B$ by $(u_0,u_0^{-1},u_0,u_0^{-1},0)$, 
 where we have denoted by  $u_0$  (respectively $u_0^{-1}$) the operator corresponding to the multiplication by $u_0$ (resp. $u_0^{-1}$). 
Observe that in particular, it follows from an integration by parts that for any  $v$ in $C^\infty_c (\R^d ; \R)$, 
\begin{equation} \label{Eq:tocomp}
- \int A u \cdot u =  \int u_0^2 | \nabla v|^2 + \mu \int   u_0 v^2 ,
\end{equation}
where the integrals are over $\R^d$ and where  $u= u_0 v$. 
The previous lines apply in particular to the case where $d=1$, $V(x) = 2 \tanh^2 x -1$, for which $A$ reduces to the first operator in \eqref{double}, with $\mu = 0$ and $u_0=\sech $.  Let us mention that  the case where $\mu =1$ is also of interest in Anderson's localization theory where 
the function $u_0$ is known as the landscape function, see  \cite{filoche}. 
\end{example}

The following example, which involves systems, is inspired by \cite[Example 3.13]{BQ} and 
 corresponds to some particular conjugate Beltrami equations.
\begin{example} \label{exBeltrami}
Let us consider the two following differential operators: 
\begin{gather*}
 A:=
 \begin{bmatrix} 
 \partial_x & - x \partial_y
 \\  \partial_y & x  \partial_x 
 \end{bmatrix}
\quad   \text{ and } \quad  
B:=
\begin{bmatrix} 
1 & 0
 \\  0 & i x  (\partial_x^2 + \partial_y^2) - \partial_y
 \end{bmatrix},
 \end{gather*}
 Then $A$ is   a two-sided golden differential transmutation of 
  $B$ by $(P,Q,R,S,0,0)$ where 
\begin{gather*}
P := 
\begin{bmatrix} 
ix & x(i \partial_x - \partial_y)  - i
 \\  -1 &  - \partial_x - i  \partial_y
 \end{bmatrix},
\quad 
 Q := 
\begin{bmatrix} 
- i \partial_x  + \partial_y & x(\partial_x + i \partial_y)  
 \\  i & - x
 \end{bmatrix},
\\ R := 
\begin{bmatrix} 
x(i \partial_x + \partial_y)  + i & 1
 \\  -x( \partial_x - i \partial_y) &   i 
 \end{bmatrix},
\quad 
S := 
\begin{bmatrix} 
-i & 1
 \\  -x( \partial_x - i \partial_y) &  - x(i \partial_x +  \partial_y) - i 
 \end{bmatrix},
\quad 
\end{gather*}
Indeed, one may observe that $PQ=QP=I_2$, that $RS=SR=I_2$
 and that $AP=RB$. 
\end{example}
%

\subsection{The transfer theorems}

 Below we establish that appropriate differential transmutations transfer some local properties such as  hypoellipticity,  unique continuation properties and some Runge-type properties.

\subsubsection{Local solvability}

We refer to \cite[Section 1.2.4]{Lerner} for the following notion. 
 \begin{defi} \label{def-local-solv}
 We say that a  linear differential operator $A$ is locally solvable if for any non-empty open subset $\Omega$  of  $\R^d$,   for any function
  function $f$ which is  $C^\infty $ in  $\Omega$, there is a distribution $u$ on  $ \Omega$  satisfying $Au=f$ in  $ \Omega$.
 \end{defi}
Differential post-transmutations transfer local solvability.
\begin{theorem} \label{th:transfert-Loc-S}
If  the linear differential operator $B$ is locally solvable 
 and $A$ is a  differential post-transmutation of the linear differential operator $B$ then
$A$   is locally solvable.
\end{theorem}
 \begin{proof}
 By assumption, $A$ is a differential post-transmutation of $B$ by a philosopher's stone 
$(P,R,S,T)$. 
 Let $f$ in $C^\infty (\Omega)$. 
 Since $Sf$ is $C^\infty $ in  $\Omega$ and $B$ is locally solvable, 
  there is a distribution $v$ on  $ \Omega$  satisfying $Bv=Sf$ in  $ \Omega$.
 Set $u := Pv + Tf$, which is a distribution on  $ \Omega$.
 Then $Au = APv + ATf = RBv + ATf= RSf + ATf=f$, by  post-catalysis and post-regeneration.
 \end{proof}

A quantitative version of Definition \ref{def-local-solv} is the following.
 \begin{defi} 
 Let $x_0$ in $\R^d$.
 We say that a  linear differential operator $A$ of order $m \in \N$ is locally solvable at $x_0$ with a loss of $\mu \geq 0$ derivatives 
  if, for every $\ell \in \R$, there is an open neighborhood  $\Omega$  of  $x_0$ in $\R^d$, such that  for any function
  function $f$ which is in the Sobolev space $H^\ell_\text{loc} (\Omega)$, there is $u$  in  $H^{\ell +m-\mu}_\text{loc} (\Omega) $ 
    satisfying $Au=f$ in  $ \Omega$.
 \end{defi}
Then the transfer of local solvability may be precised as follows. 
\begin{theorem} \label{th:transfert-Loc-S-quanti}
If  the linear differential operator $B$ of order $m_B$ is locally solvable at $x_0$ with a loss of $\mu_B \geq 0$ derivatives 
 and the  linear differential operator 
$A$  of order $m_A$ 
 is a differential post-transmutation of the linear differential operator $B$  by a philosopher's stone 
$(P,R,S,T)$ then
$A$  is locally solvable at $x_0$ with a loss of 
$\mu_A := m_A - \inf(m_B -s-\mu_A - p, -t)$
derivatives 
 where $s$, $p$ and $t$ are respectively the differential orders of the operators $S$, $P$ and $T$.
\end{theorem}
 \begin{proof}
 For any $\ell \in \R$, there is an open neighborhood  $\Omega$  of  $x_0$ in $\R^d$, such that  for any function
  function $f$ which is in the Sobolev space $H^\ell_\text{loc} (\Omega)$, there is $v$  in  $H^{\ell+m_B -s-\mu_B}_\text{loc} (\Omega) $ 
    satisfying $Bv=Sf$ in  $ \Omega$.
 Set $u := Pv + Tf$, which is in $H^{\ell + \inf(m_B -s-\mu_B - p,-t)}_\text{loc} (\Omega)$. 
 Moreover, as previously,  
$Au =f$.
 \end{proof}


\subsubsection{Differential inverses}

 \begin{defi} \label{def-diff-inv}
 We say that a  linear differential operator $A$ admits a right (respectively left) differential inverse  $M$ (resp. $L$) if 
 for any function $f$ which is  $C^\infty $ in  $\R^d$,  $AMf=f$ in  $ \R^d$, (resp. $LAf=f$ in  $ \R^d$).
 \end{defi}
Differential post-transmutations transfer right differential inverses while differential pre-transmutations transfer right differential inverses.
\begin{theorem} \label{th:transfert-dii-inv}
If  the linear differential operator $B$ admits  a right (respectively left) differential inverse  $M$ (resp. $L$) and
$A$ is a differential post-transmutation (resp. pre-transmutation) of the linear differential operator $B$ 
 by a philosopher's stone 
$(P,R,S,T)$ (resp. $(P,Q,S,T)$) 
then
$A$ admits $PMS+T$ (resp. $PLS+T$) as  a right (respectively left) differential inverse. 
\end{theorem}
 \begin{proof}
On the one hand, $A(PMS+T)=RBMS+AT= RS+AT=\Id$ by post-catalysis, definition of $M$ and post-regeneration.
On the other hand, $(PLS+T)A=PLBQ+TA=PQ+TA=\Id$ by pre-catalysis, definition of $L$ and pre-regeneration.
 \end{proof}
Observe that the second statement can also be deduced from the first one by taking the adjoint and using  Proposition \ref{prop-elem8}.

\subsubsection{Hypoellipticity}

Let us first recall the following notion of hypoellipticity. 
 \begin{defi} \label{def-hypoelliptic}
 A linear differential operator $A$ is hypoelliptic 
if for any non-empty open subset $\Omega$  of  $\R^d$,   for any function $f$ which is  $C^\infty $ in  $\Omega$, 
  for any distribution $u$ on  $ \Omega$  satisfying $Au=f$ in  $ \Omega$, then 
 $u$ is   $C^\infty $ in $ \Omega$. 
\end{defi}

Differential pre-transmutations transfer hypoellipticity. 
\begin{theorem} \label{th:transfert-HYPO}
If  $A$ is a differential pre-transmutation of  $B$ and if $B$ is hypoelliptic,
 then  $A$ is hypoelliptic. 
\end{theorem}
 \begin{proof}
Consider  a  distribution $u$ on  $ \Omega$  satisfying $Au=f$ in  $ \Omega$.  
By assumption, $A$ is a differential pre-transmutation of $B$ by a philosopher's stone 
$(P,Q,S,T)$. 
We introduce the distribution  $v:=Qu $ on $ \Omega$. 
Then, by pre-catalysis 
$Bv = BQu =SAu
 =Sf$, and 
since $Sf$ is $C^\infty $ on  $\Omega$ and $B$ is hypoelliptic, 
 $v$ is  $C^\infty $ on $ \Omega$.  
 Moreover, by pre-regeneration, 
 $u = PQ u + TAu
 = Pv+ Tf$ 
 is also  $C^\infty $ on $ \Omega$. 
 \end{proof}
Let us highlight that this result can be easily extended to $C^\omega$ hypoellipticity or to Sobolev hypoellipticity. In the latter, one infers from the proof above that the shift $m_A$ 
in the Sobolev scale by the operator $A$ satisfies the lower bound $m_A \leq \inf(m_B - s - p,t)$, where $m_B$ is the shift in the Sobolev scale by the operator $B$ and $s$,  $p$ and $t$ are respectively the differential orders of the operators $S$, $P$ and $T$.
\medskip

In a similar way, we may define some  \textit{a priori} estimates in the spirit of 
\cite[Lemma 1.2.30]{Lerner} and prove that they are transferred by differential pre-transmutations. 
Combined with Proposition \ref{prop-elem8} and the duality between these   \textit{a priori} estimates and the local solvability, this provides a second proof of Theorem \ref{th:transfert-Loc-S}.

\subsubsection{Unique continuation}

Let us now turn to unique continuation. 
Let us recall the following definition of the weak unique continuation property. 
 \begin{defi} \label{def-WUC}
 A linear differential operator $A$ satisfies the weak unique continuation property if for any  non-empty open subset $\tilde \Omega$,
  for any non-empty open subset $\Omega$  of $\tilde \Omega$,  
  for any smooth function $u$ defined on  $\tilde \Omega$ satisfying $Au=0$ in  $\tilde \Omega$ and $u=0$ in $ \Omega$, then  $u=0$ in $\tilde \Omega$.
\end{defi}
Differential pre-transmutations transfer the weak unique continuation property. 
\begin{theorem} \label{th:transfert-WUC}
If  $A$ is a differential pre-transmutation of the linear differential operator $B$ and if $B$ satisfies the  weak unique continuation property, then 
 $A$ satisfies   the weak unique continuation property. 
\end{theorem}
 \begin{proof}
Consider a non-empty open subset $\tilde \Omega$ and  a smooth function $u$ defined on  $\tilde \Omega$ satisfying $Au=0$ in  $\tilde \Omega$ and $u=0$ in  a any non-empty open subset $\Omega$  of $\tilde \Omega$. 
By assumption, $A$ is a differential pre-transmutation of $B$ by a philosopher's stone 
$(P,Q,S,T)$.
Set $v:=Qu $ on $\tilde \Omega$. Clearly $v=0$ in $ \Omega$. 
Moreover, by pre-catalysis, 
 $Bv = BQu =SAu =0$  on $\tilde  \Omega$. 
 Since $B$ satisfies the  weak unique continuation property,  
$v=0$ in $\tilde  \Omega$. 
Finally, by pre-regeneration,  $u = Pv= 0$ in $\tilde  \Omega$, 
\end{proof}
We now turn to the strong unique continuation. 
 \begin{defi} \label{def-SUC}
 A linear differential operator $A$ satisfies the strong unique continuation property if for any non-empty open subset $\Omega$  of $\R^d$, 
 for any $x_0$ in $\Omega$, 
  for any smooth function $u$ defined on  $\Omega$ satisfying $Au=0$ in  $ \Omega$ 
  and  that there exists $R > 0$ such that, for any $N \in \N$, 
  there exists $C_N > 0$, for any $r$ in $(0,R)$, 
  \begin{equation}
  \label{infini}
 \int_{B(x_0,r)}  \vert  u \vert^2 \, dx \leq C_N \, r^N ,
\end{equation}
then  $u=0$ in $ \Omega$. 
\end{defi}
Differential pre-transmutations transfer the strong unique continuation property toward elliptic operators.
\begin{theorem}\label{th:transfert-SUC}
If  an elliptic linear differential operator $A$ is a differential pre-transmutation of the  linear differential operator $B$ and if $B$ satisfies the strong unique continuation property, then 
 $A$ satisfies   the strong unique continuation property. 
\end{theorem}
 \begin{proof}
Consider  a non-empty open subset $\Omega$ of $\R^d$, a smooth function $u$ defined on  $ \Omega$ satisfying  $Au=0$ in  $ \Omega$,  $x_0$ in $\Omega$, $R > 0$ and  $(C_N )_N$ in $\R_+^*$ 
  such that 
   \eqref{infini} holds true. 
   By assumption, $A$ is a  differential pre-transmutation of $B$ by a philosopher's stone 
$(P,Q,S,T)$.
 Denote by $m \in \N$ the degree of $Q$, that is the maximal number of derivatives of the  linear differential operator $Q$. 
 Set $v:=Qu $ in $\Omega$. 
 Since $A$ is elliptic, by interior regularity and a scaling argument, there is $C >0$ such that for any $r$ in $(0,R/2)$, 
  \begin{equation}
  \label{}
 \int_{B(x_0,r)}  \vert  v \vert^2 \, dx \leq C r^{-2m}   \int_{B(x_0,2r)}  \vert  u \vert^2 \, dx     .
\end{equation}
Combining with the inequality  \eqref{infini}, with $2r$ instead of $r$ and $N+2m$ instead of $N$, we deduce that, for any $r$ in $(0,R/2)$, for any $N \in \N$, 
  \begin{equation*}
 \int_{B(x_0,r)}  \vert  v \vert^2 \, dx \leq \tilde C_N \,   r^{N}  ,
\end{equation*}
with $\tilde C_N := C C_{N+2m}  \, 2^{N}$.
But, by pre-catalysis 
 $Bv=0$ in $\Omega$, and 
  $B$ satisfies the  strong unique continuation property; so we deduce that $v=0$ in $  \Omega$. 
By pre-regeneration,    $u= Pv=0$ in $ \Omega$. 
\end{proof}

One may also prove that pre-transmutations transfer the unique continuation property across a hypersurface, which is the property that any solution of a PDE in a neighborhood of a hypersurface, which vanishes on one side the hypersurface, also vanishes on the other side, see  \cite[Chapter 5]{LLR} for more.
\smallskip

On the other hand, since, as it is, our notion of transmutation does not take care of the issue of boundary conditions, 
 it is not appropriate for transferring  the unique continuation property for  local Cauchy data. 
 However, an equivalence theorem, due to Lax,  \cite{Lax}, states that the latter is equivalent to a Runge-type property, which we investigate below under a strong form.

\subsubsection{The Runge property}

We recall the following classical notion in approximation theory.

 \begin{defi} \label{RungeP}
 A linear differential operator $A$ satisfies the Runge property with point controls on  $\R^d$ if for any compact $K$ 
 for any  function $f$ in the image of $C^\infty(\R^d)$,  by $A$, 
 for any integer $k$, for any $\eps >0$, for any set  $E$ 
 with exactly one point  in each bounded connected component of  $\R^d \setminus K$, for any function $u$ defined on 
 a open neighbourhood $\Omega$ of $K$ and satisfying $Au=f$ in  $\Omega$, then there exists a function $ \overline u $ defined on the whole space $\R^d$ such that 
$A \overline u = f + g$
where $g$  is a linear combination of derivatives of Dirac masses located in $E$ and $ \|  \overline u - u \|_{C^k (K) }   \leq \eps$. 
\end{defi}
Let us mention that despite the case where  $f$ is the null function straightforwardly implies the general case, we kept this source term $f$ to exhibit the generality of the property.
\smallskip

Observe that the set $E$ can be infinite, as for example in the case of Mergelyan-Roth's swiss cheese, see 
   \cite[20, No. 2.4]{Mergelyan} and  \cite[Section 1, Hilfssatz 4]{Roth}. However the combination of derivatives of Dirac masses involved in Definition \ref{RungeP}
   only needs to be finite.
\smallskip

Let us also highlight that in the  particular case where $\R^d \setminus K$ is connected, then $E=\emptyset$ and the conclusion in the case where  $f$ is the null function 
 is that there exists   a function $ \overline u $ defined on the whole space $\R^d$ such that  $A \overline u =0$  in $\R^d$ and $ \|  \overline u - u \|_{C^k (K) }   \leq \eps$. 
\smallskip

This property is named after Runge who considered the case where $A = \overline{\partial}$ in  \cite{Runge}. Later this analysis has been extended to the Laplace operator, in any positive dimension, by Walsh in  \cite{Walsh}, see also 
\cite{DP,DGO,GARD}. 
\begin{theorem}\label{th:approxharm1}
The Laplace operator satisfies  the Runge property with point controls on  $\R^d$, for any $d \in \N^*$.
\end{theorem}
In the case of $2$D
 finitely connected region, with $f=0$,  this result  can be straightforwardly transferred  from the  Runge result on holomorphic functions by using the logarithmic conjugation theorem, see the appendix. 
\medskip

Let us also mention that such results have been extended to a wide class of elliptic operators by Lax in \cite{Lax}, Malgrange in \cite{Malgrange} and Browder in \cite{Browder}. 
The proofs of these results split into two categories: some based on an integral representation of the local solution, a discretization into a finite Riemann sum and a final poles' pushing step; the others are based on the duality with unique continuation properties associated with data on a hypersurface, after Holmgrem's and Carleman's approaches. 
\medskip
 
Silvern  differential transmutations  transfer the Runge property with point controls.
\begin{theorem}\label{th:transfert}
If  the linear differential operator $A$ is a silvern differential transmutation of the  linear differential operator $B$ and if $B$ satisfies  the Runge property with point controls on  $\R^d$, then 
 $A$ satisfies  the Runge property with point controls on  $\R^d$. 
\end{theorem}
 \begin{proof}
Consider a compact $K$ 
an integer $k$, a positive real $\eps $, a set  $E$ 
 with exactly one point  in each bounded connected component of  $\R^d \setminus K$, a function $f$ in the image of $C^\infty(\R^d)$  by $A$ 
 and a function $u$ defined on 
 a open neighbourhood $\Omega$ of $K$  satisfying $Au=f$ in  $\Omega$.
 By assumption, $A$ is a silvern differential transmutation of $B$ by a philosopher's stone 
$(P,Q,R,S,T)$.
  Denote by $m \in \N$ the degree of $P$, that is the maximal number of derivatives. 
  Set $v:=Qu $ on $\Omega$. 
Then, on $\Omega$, by pre-catalysis, 
we have that $Bv = Sf$. 
Since $f$ is in the image of $C^\infty(\R^d)$  by $A$, by pre-catalysis, $Sf$ is in the image of $C^\infty(\R^d)$  by $B$.
Since $B$ satisfies  the Runge property with point controls on  $\R^d$,  there exists  $ \overline v $ defined on the whole space $\R^d$ such that 
$B \overline v = Sf + h$ where 
 $h$ is a linear combination of derivatives of Dirac masses located in $E$ and $ \|  \overline v - v\|_{C^{k+m}  (K) }   \leq \eps$. 
Set $\overline u := P \overline v + Tf$  and $g := Rh$ on  $\R^d$.
Since  $R$ is a linear differential operator with $C^\infty$  coefficients,  $g$  is also  a linear combination of derivatives of Dirac masses in  $E$. 
Moreover,  by post-catalysis, 
$A \overline u = RB \overline v + ATf= (RS+AT)f + Rh$.
By Proposition \ref{prop-post-r}, 
 $A$ is weakly post-regenerated by $(R,S,T)$ 
 and therefore $A \overline u = f + g$. 
 On the other hand,  by pre-regeneration,  $u =  Pv + Tf$  on $\Omega$  so that 
 $ \|  \overline u - u \|_{C^k (K) }   \leq C \|  \overline v - v\|_{C^{k+m}  (K) }   \leq C \eps$, where $C>0$ depends only on the coefficients of $P$.
 Therefore $A$ satisfies  the Runge property with point controls on  $\R^d$. 
\end{proof}

Let us highlight that quantitative versions of the  Runge property, which evaluate the 
cost of approximation, can be transferred too. 
In this direction, let us consider the rate of convergence by polynomial solutions. 
Given a linear differential operator $A$,  for any $n$ in $\N$, let $\mathcal P_n (A)$ denote the space of the polynomials $p$ of order less than $n$ which satisfy $Ap = 0$ on $\R^d$.

 \begin{defi} \label{poly-Runge}
 A linear differential operator $A$ satisfies the polynomial Runge property on  $\R^d$ 
  if for any compact $K$ of $\R^d$ such that $\R^d \setminus K$ is connected,  for any open neighbourhood $\Omega$ of $K$, there exists $\rho \in (0,1)$ such that 
 for any integer $k$, for any function $u$ defined on  $\Omega$ and satisfying $Au=0$ in  $\Omega$, 
  then there exists $C>0$ such that for any $n$ in $\N$, 
  $$ \inf_{p \in \mathcal P_n (A)} \,  \| u - p \|_{C^{k} (K)}  \leq C \rho^n .$$
 \end{defi}
The following result can be found in  \cite{Walsh} for the $2$D case and 
in \cite{And} and \cite[Theorem 3.1]{BL} for the general case. 
\begin{theorem}\label{th-lap-poly}
For any $d \in \N^*$,
 the Laplace operator on $\R^d$ satisfies  the  polynomial Runge property on  $\R^d$. 
\end{theorem}

Theorem \ref{th:transfert} can be adapted to  transfer the polynomial Runge property in the case where the first auxiliary differential operator $P$ in the philosopher's stone  
 has polynomial coefficients. 

\begin{theorem}\label{th:transfert-poly}
If  the linear differential operator $A$ is a silvern differential transmutation of the  linear differential operator $B$  by  
$(P,Q,R,S,T)$ where the first stone $P$ has polynomial coefficients 
and if $B$ satisfies  the polynomial  Runge property on  $\R^d$, then 
 $A$ satisfies  the polynomial  Runge property on  $\R^d$. 
\end{theorem}
 \begin{proof}
  Denote by $\ell \in \N$ the maximal number of derivatives and by $m$ the maximal power of $x$ of $P$. 
  Let  $K$ be a compact of $\R^d$ and $\Omega$ an open neighbourhood  of $K$. 
  Since $B$ satisfies  the polynomial  Runge property on  $\R^d$, there exists $\rho \in (0,1)$ such that 
 for any integer $k$, 
 for any function $v$ defined on $\Omega$ of $K$ and satisfying $Bv=0$ in  $\Omega$, 
  then there exists $C>0$ such that for any $n$ in $\N$, 
$$
\inf_{p \in \mathcal P_n (B)} \,  \| v - p \|_{C^{k} (K)}  \leq C \rho^n .
$$
Consider now  any integer $k$ and  any function $u$ defined on  $\Omega$ and satisfying $Au=0$ in  $\Omega$.
 By pre-catalysis, $Qu$ satisfies $BQu=0$ in  $\Omega$. Thus, by what precedes, 
 there exists $C_1>0$ such that for any $n$ in $\N$, 
  \begin{equation}
  \label{GA}
   \inf_{q \in \mathcal P_n (B)} \,  \| Qu - q \|_{C^{k+\ell}(K)}  \leq C_1 \rho^n .
   \end{equation}
By post-catalysis, for any $n$ in $\N$, we have that $P \mathcal  P_n (B) \subset \mathcal  P_{n+m} (A)$.
 Moreover,  by pre-regeneration,  $u =  PQu$  on $\Omega$  so that 
  $$ \inf_{p \in \mathcal P_{n+m} (A)} \,  \| u - p \|_{C^{k} (K)}  \leq 
  C_2   \inf_{q \in \mathcal P_{n} (B)} \,  \| Qu - q \|_{C^{k+\ell}(K)},$$
 where $C_2$ only depends on the coefficients of $P$ and of $K$. Therefore, by \eqref{GA}
  $$ \inf_{p \in \mathcal P_{n+m} (A)} \,  \| u - p \|_{C^{k} (K)}  \leq 
  C  \rho^{n+m} ,$$
 where $C \geq C_1 C_2 \rho^{-m} $,  for any $n$ in $\N$, which concludes the proof. 
 \end{proof}
Let us highlight that in the proof above the rate $\rho$ is preserved during the transfer from $B$ to $A$.

 Let us also mention that one may also transfer some quantitive estimates in the spirit of 
  \cite{RS} with similar reasoning.
\medskip

That poles are mandatory in the domain's holes is highlighted by the following lines.
 \begin{defi} \label{pole-less} 
 We say that a linear differential operator $A$  satisfies the pole-less Runge property  for  a pair of open subsets $\Omega$ and  $\tilde \Omega$ of $\R^d$ with 
 $\Omega \subset \tilde  \Omega$  if 
 for any compact  $K \subset \Omega$ 
 for any integer $k$, for any $\eps >0$, for any function $u$ defined on 
 a open neighbourhood $\Omega$ of $K$ and satisfying $Au=0$ in  $\Omega$,  there exists a function $ \overline u $ defined on the whole space $\R^d$ such that 
$A \overline u = 0$ and $ \|  \overline u - u \|_{C^k (K) }   \leq \eps$. 
\end{defi}
 \begin{defi} \label{detection} 
 We say that a linear differential operator $A$ has the detection property of simply connected complements 
  if  for any open subset $\Omega$ of $\R^d$, for any open set $\tilde \Omega$ of $\R^d$ with 
 $\Omega \subset \tilde  \Omega$, if $A$ satisfies the pole-less Runge property then 
  $\tilde  \Omega \setminus  \Omega$ is simply connected. 
 \end{defi}
It is well-known that, as a consequence of  the maximum principle, we have the following result for the Laplace operator, in any positive dimension, see for instance \cite{AG}. 
\begin{theorem}\label{11}
The Laplace operator has the detection property of simply connected complements, for any dimension $d \in \N^*$.
\end{theorem}
Moreover the  detection property of simply connected complements can be transferred backward by silvern differential transmutations.
\begin{theorem}\label{th:transfert-11}
If  the linear differential operator $A$ is a silvern differential transmutation of the  linear differential operator $B$ and if $A$ has  the detection property of simply connected complements, then  $B$  also has  the detection property of simply connected complements.
\end{theorem}
 \begin{proof}
Consider  a pair of open subsets $\Omega$ and  $\tilde \Omega$ of $\R^d$ with 
 $\Omega \subset \tilde  \Omega$, for which $B$ satisfies the pole-less Runge property. 
 By assumption, $A$ is a silvern differential transmutation of $B$. 
 Following the proof of Theorem \ref{th:transfert}, we obtain that 
 $A$ also satisfies the pole-less Runge property for $\Omega$ and  $\tilde \Omega$. 
By assumption, $A$ has  the detection property of simply connected complements, so  $\tilde  \Omega \setminus  \Omega$ is simply connected. 
This establishes that  $B$  also has  the detection property of simply connected complements.
\end{proof}

The following table recapitulates which type of transmutation allows to transfer which properties.  
$$ 
\begin{tabular}{| c | c |}
 \hline	    \text{Post-transmutation} & Local solvability  \\ 
    & Differential right inverse  \\
     \hline 		
  \text{Pre-transmutation} & Hypoellipticity \\
  & Differential left inverse  \\
  & Weak unique continuation property \\
  & Unique continuation property across a hypersurface \\
  & Strong unique continuation property if $A$ is elliptic  \\  \hline
  \text{Silvern} &  Runge property with point controls \\
  & Polynomial Runge property    \\  \hline
   \text{Backward silvern} & Detection property of simply connected complements     \\  \hline 
 \end{tabular}
$$
\medskip

\medskip

\section{The $2$D Lamé-Navier operator}
\label{sec-2D}

This section is devoted to the example of the  $2$D Lamé-Navier operator. 
Thus we consider the case where  $d=2$. Let $\nu \in \R$   
 and consider the case where the operator $A$ is the differential operator which maps $(u,p)$ 
  in  $C^\infty  (\R^2 ; \R^2 \times \R)$ as follows: 
\begin{equation} \label{Eq:LM-Op}
\mathfrak L_\nu : (u,p) \mapsto ( \Delta u - \nabla p ,  \div u - \nu p)  . 
\end{equation}  
One readily sees that, for any  $\nu \in \R$, the operator $\mathfrak L_\nu$ is formally self-adjoint, 
 and for any $(u,p)$ 
  in  $C^\infty_c  (\R^2 ; \R^2 \times \R)$, 
\begin{equation} \label{Eq:q}
- \int \mathfrak L_\nu (u,p) \cdot (u,p) =  \int \Big(  \nabla u : \nabla u  -   2 p \div u    + \nu p^2  \Big),
\end{equation}  
where the integrals are over $\R^2$. 
In the case where $\nu=0$, the operator $\mathfrak L_\nu$ above corresponds to the 
  steady Stokes system, which  aims  at describing steady, incompressible fluids with zero-Reynolds number. %
In particular $u$ stands for  the fluid velocity, which is vector-valued, and  $p$ for the fluid pressure which is scalar-valued.  We refer to \cite{Galdi}  for more on this system.
The other values of $\nu$ are also of interest in linear elasticity, since then the operator $\mathfrak L_\nu$ above corresponds to the
 Lam\'e-Navier system, which aims  at describing the displacement $u$ in an elastic material, see  \cite{Gurtin}  for more.
 The latest usually reads 
 \begin{equation}
  \label{LN-eq}
(\lambda + \mu) \nabla \div u + \mu \Delta u = 0  ,
\end{equation}
  where $\lambda$ and $\mu$ are the Lam\'e constants; so that it is sufficient to set 
 $$\nu := - \frac{\mu  }{\mu+\lambda }  \,  \text{ and } \,
 p := - \frac{\mu +\lambda  }{\mu} \div u ,$$ 
 to make appear the operator $\mathfrak L_\nu$  given by  \eqref{Eq:LM-Op}. 
 However an important difference, compared to the  steady Stokes system, is that, for $\nu \neq 0$, for a solution $(u,p) $ of $\mathfrak L_\nu (u,p) =0$, one has a local expression of $p$ in terms of $u$.
 
 %
\subsection{The $2$D Lamé-Navier operator as a two-sided golden differential transmutation of $\Delta^{\otimes 2}$}
 
In this section we consider  the case where the operator $B$ is 
\begin{equation} \label{B2D}
B= \Delta^{\otimes 2},
\end{equation}  
 where, for any $\ell \in \N^*$ we  define the Laplace operator $\Delta^{\otimes \ell} $  
 as the operator   which maps  functions in
 $C^\infty (\R^d ; \R^\ell )$ to functions in  $C^\infty (\R^d ; \R^\ell)$  by the following formula:
\begin{equation} \label{Eq:tensor}
\Delta^{\otimes \ell}: v:=(v_i )_{1 \leq i \leq \ell}  \mapsto ( \Delta v_i)_{1 \leq i \leq \ell} .
\end{equation}
Let us highlight that we freely use the simpler notation $\Delta$ for $\Delta^{\otimes 2}$ in vector identities when there is no risk of confusion, as we actually already did in 
 \eqref{Eq:LM-Op}.

For any $\nu \in \R \setminus \{2^{-1}  ,1\}$, we define  $\lambda_\nu$ by 
\begin{equation}
  \label{cho}
\lambda_\nu^{-1} :=  {4(1-2 \nu)(1-\nu)}  .
\end{equation}
We define the following linear differential operators acting on functions $v$ in  $C^\infty  (\R^2 ; \R^2)$  by  
\begin{align}
\label{defUetP2D} 
  P v &:=  \Big((1-2 \nu) v -\frac12 x \,  \div v + \frac12 x^\perp \,  \rot v   , -2\div v \Big) , 
\\  \label{defR2D}  
 R v &:=  \Big((2-2 \nu) v  -\frac12 x \,  \div v + \frac12 x^\perp \,  \rot v ,   -\frac12  x \cdot  v  \Big),
 \\ \label{choT}  \tilde  T v &:=   \frac{1}{2}  \lambda_\nu   | x |^2  v,
 \\  \check Tv &:=    \frac{1}{2} (x \cdot v) x  .
\end{align}
Above we omit to write that the functions are evaluated in $x$ for sake of clarity and we denote by $\rot $ the 2D rotational, defined by 
$\rot v = \partial_1 v_2 - \partial_2 v_1$, for a vector field  $v=(v_1,v_2)$. 
Also we denote by $x^\perp = (-x_2 ,x_1)$ the rotation by $+\pi /2$ of the vector $x=(x_1,x_2)$. 
We also set, for functions $(u,p)$ 
  in  $C^\infty  (\R^2 ; \R^2 \times \R)$, 
\begin{align*}
T (u,p) &:=  4\lambda_\nu \Big( \frac12  p x + \frac18 u  | x |^2  , (3-2\nu) p +  \frac12 x  \cdot u \Big) ,
 \\  \hat   Q (u,p) &:=   x \div u - \nu p x  ,
 \\  \hat   S(u,p) &:=   x\Delta p + 2 \nabla p ,
\\  \hat   T(u,p) &:=  \Big( \frac12 | x |^2 \nabla p + 2 \nu p x  , 4 p +  2 x \cdot \nabla p \Big)  .
\end{align*}
Let us observe that $T$ and $\tilde T$ are formally self-adjoint.
\begin{theorem}\label{th:main2D}
For any $\nu \in \R \setminus \{2^{-1} ,1\}$, 
for any $\lambda $ in $\R$, 
$\mathfrak L_\nu$ is a  two-sided golden differential transmutation of 
$ \Delta^{\otimes 2}$ by philosopher's stone $(P,2 \lambda_\nu   R^* +\lambda \hat Q,R ,2 \lambda_\nu   P^* +\lambda \hat S,T+\lambda \hat T,\tilde T+\lambda  \check T)$.
\end{theorem}
Observe that the forbidden value $\nu=1$ corresponds to the critical value for which strong ellipticity of $\mathfrak L_\nu$ is lost, see \eqref{Eq:q}.
 Let us refer to  Section \ref{transE} for another insight on the two forbidden values $\nu \in  \{2^{-1} ,1\}$.

\begin{proof}
According to   Proposition  \ref{prop-gauge-2-sided} and  Proposition \ref{prop-t}, with  $\lambda = (2  \lambda_\nu )^{-1}$, 
 it is sufficient to prove that  
 $\mathfrak L_\nu P=R\Delta$, $ 2 \lambda_\nu   PR^* +T\mathfrak L_\nu =\Id$, $2 \lambda_\nu   R^* P+\tilde T \Delta = \Id$, 
$\Delta\hat  Q  = \hat  S \mathfrak L_\nu$, 
$ P \hat Q = - \hat T \mathfrak L_\nu$, $\hat Q P = - \check T \Delta$,
 $R  \hat  S = - A  \hat  T$ 
 and 
 $  \hat  S R = - \Delta  \check  T$. 
 These identities can be checked by direct computations.  \end{proof}

\subsection{Consequences for local properties}
\label{cons2D}
  As a direct consequence of 
  Theorem \ref{th:transfert-HYPO},
Theorem \ref{th:transfert-WUC}, 
Theorem \ref{th:transfert-SUC}, 
 Theorem \ref{th:transfert}, Theorem \ref{11}, Theorem \ref{th:transfert-11}, Theorem \ref{th:main2D}, Theorem \ref{th:approxharm1},  Theorem \ref{th-lap-poly} and Theorem \ref{th:transfert-poly}
 we have the following result.
\begin{corollary}\label{th:rw}
The operator $ \mathfrak L_\nu$, defined in  \eqref{Eq:LM-Op} with $\nu \in \R \setminus \{\frac12 ,1 \}$, satisfies the weak unique continuation property, 
  the Runge property with point controls, the polynomial Runge property  on  $\R^2$, and the detection property of simply connected complements. 
  Moreover, if $\nu>1$, then  the  operator $ \mathfrak L_\nu$ satisfies the strong unique continuation property. 
\end{corollary}
We postpone comments until Section 
\ref{cons3D} where similar results are obtained in the 3D case, with a comparison to some earlier results which hold in any dimension.

\subsection{Generation of  fundamental solutions}

As another corollary of Theorem \ref{th:main2D}, 
we  have the following  result regarding the 
generation of  fundamental solutions of  the  Lamé-Navier   and Stokes systems by the means of 
fundamental solutions of  the Laplace problem
(compare to Example \ref{excar}).
\begin{corollary} \label{th-fonda2D}
Let $b$ in $\R^2$.  
Let $\Omega$  a domain in $\R^2$ with $0 \in \Omega$.
 \begin{enumerate}[leftmargin=19pt]
\item   Let  $v:\Omega \rightarrow \R^2$  satisfies 
$\Delta v   = b \delta_0$ and  set $ (u,p) :=  Pv$.
Then,  for any $\nu \in \R $, 
 \begin{equation} \label{2b2D}
\mathfrak L_\nu(u,p)= ((2-2\nu) b \delta_0 ,0), \quad  \text{ in }  \quad  \Omega. 
\end{equation}
\item Let $\nu \in \R \setminus \{\frac12 ,1 \}$. 
Let $u:\Omega \rightarrow \R^2$ and $p:\Omega \rightarrow \R$
satisfying    \eqref{2b2D}. 
Set $v := Q(u,p)$, 
Then $\Delta v   = b \delta_0$ and  $ (u,p) :=  P v$.   
\end{enumerate}
\end{corollary}
\begin{proof} We start with the proof  of $(1)$. 
By elementary computations,  $x^\perp  \rot (b  \delta_0) =  -  b \delta_0 $ and $x \, \div (b  \delta_0)  = -  b \delta_0$, 
so that $R b  \delta_0  = ((2-2\nu)b  \delta_0 ,0)$.
Then, by post-catalysis,   $\mathfrak L_\nu(u,p)= R( b  \delta_0 , 0)  = ((2-2\nu)b  \delta_0 ,0)$.
To prove $(2)$, we use   pre-catalysis to get 
$\Delta v =S  \mathfrak L_\nu (u,p) = S((2-2\nu)b  \delta_0 ,0)= b \delta_0$.
Moreover
 $T((2-2\nu)b  \delta_0 ,0 ) = 0$ so that $(u,p) = Pv$ by pre-regeneration.
\end{proof}
In the  peculiar case of Corollary  \ref{th-fonda2D}
where $b=0$, we recover some earlier results on the representation,  in an arbitrary domain, of solutions to the $2$D steady Stokes system  and to the $2$D  Lamé-Navier  system 
respectively obtained  by Kratz in \cite{Kratz}  (in the case where $\nu =0$) and 
 by Zsupp\'an in \cite{Zsuppan} (in the case where  $\nu \in \R \setminus \{\frac12 ,1\}$). 
 These results are interesting variations on the theme of the so-called Neuber-Papkovich potentials, see 
\cite{Neuber} and \cite{Papkovich}, and \cite{Gurtin} for a detailed exposition. 
 \medskip

\subsection{Transfer of energy}
\label{transE}

Next result concerns the energy-type quantity  associated to $\mathfrak L_\nu$ which is defined in  \eqref{Eq:q}, for vector fields $(u,p) $ in the image of the operator $P$ (compare to \eqref{Eq:tocomp}). 
 \begin{proposition} \label{energyt}
 When the operator $\mathfrak L_\nu$ is given by  \eqref{Eq:LM-Op} with $\nu \in \R$ and $P$ is given by  
 \eqref{defUetP2D}, then 
 for any  $v$ in $C^\infty_c (\R^2 ; \R^2)$, 
\begin{equation} \label{Eq:qqq}
- \int \mathfrak L_\nu (u,p) \cdot (u,p) = 
\frac{ 1 }{2 \lambda_\nu  } \int | \nabla v |^2 + \frac14 \int    | x |^2 | \Delta v |^2 ,
\end{equation}  
where $(u,p) = Pv$. 
\end{proposition}
One observes that, while the second term of the right hand side of \eqref{Eq:qqq} is nonnegative, 
the sign of the  first term
 changes at the two critical values 
$\nu = 1/2$ and  $\nu =1$, and is positive in the case where $\nu=0$.
For $\nu \in \R \setminus [\frac12 ,1 ]$, we deduce from Proposition \ref{energyt} the following Liouville-type result: 
if  $v$ in $C^\infty_c (\R^2 ; \R^2)$ satisfies $\mathfrak L_\nu Pv=R\Delta v=0$ then $v=0$. 
\begin{proof}
For any  $v$ in $C^\infty (\R^2 ; \R^2)$, 
the quantity in \eqref{Eq:q} with $(u,p) = Pv$  is then 
\begin{equation*} 
- \int \mathfrak L_\nu (u,p) \cdot (u,p) =  - \int  P^* \mathfrak L_\nu P v \cdot v  
= - \int  P^* R \Delta v \cdot v, 
\end{equation*}  
by post-catalysis. 
By transposition of the identity: $2 \lambda_\nu   R^* P+\tilde T \Delta = \Id$, we have 
 $2\lambda_\nu   P^* R  + \Delta  \tilde T  = \Id$, 
 and thus, 
\begin{equation*}
- \int \mathfrak L_\nu (u,p) \cdot (u,p)   
= -\frac{ 1 }{2 \lambda_\nu  }  \int  \Delta v \cdot v  + \frac{ 1 }{2 \lambda_\nu  }   \int  \Delta  \tilde T  \Delta v \cdot v , 
\end{equation*}  
which, by some  integration by parts, and  recalling  
  \eqref{choT}, leads to \eqref{Eq:qqq}. 
 \end{proof}


\section{The $3$D steady Stokes operator}
\label{sec-3D}

In this section we again consider the operator  defined in  \eqref{Eq:LM-Op}  
 but this time set on $3$D domains.

\subsection{The $3$D steady Stokes operator as a golden differential transmutation of $\Delta^{\otimes 4}$}

Recalling the notation  \eqref{Eq:tensor}, we set 
\begin{equation}
  \label{B3D}
B:=\Delta^{\otimes 4}.
\end{equation}
We  define the two linear differential operators with variable coefficients: $P$ and $Q$, defined by their action on $C^\infty$ functions $(v,w)$  with values in $\R^3 \times \R $ by 
\begin{align*} 
P(v,w)  &:=   \Big( \frac32 v -\frac12 x \div v - \frac12  x\wedge (\curl v - \nabla  w)  , -2\div v \Big)  , 
 \\  
  R (v,w) &:= \Big( \frac52 v -\frac12   x \div v  - \frac12  x\wedge (\curl v- \nabla  w)    , -\frac12  x \cdot   v  \Big).
    \end{align*}
  For $(u,p)$  in  $C^\infty  (\R^3 ; \R^3 \times \R)$, we define 
  $$  T(u,p) :=   \Big( \frac{1}{3}   px +  \frac{1}{12} |x|^2  u  ,   \frac{7}{3} p +   \frac{1}{3} x \cdot u \Big), $$
  which  is formally self-adjoint, and the gauge operators:
\begin{align*}
\hat   Q (u,p) &:=  ( x\div u   ,0),
 \\  \hat   S(u,p) &:=  ( x\Delta p + 2 \nabla p , 0),
\\  \hat   T(u,p) &:=(  \frac12 | x |^2 \nabla p  , 6 p +  2 x  \cdot \nabla p) .
\end{align*}

Let  $\mathfrak L_0$ be the $3$D steady Stokes operator, that is the operator   defined in  \eqref{Eq:LM-Op}   on $\R^3$ 
 with $\nu =0$. 
\begin{theorem}\label{th:main3D2}
For any $\lambda \in \R$,
 $\mathfrak L_0 $ is a golden  transmutation of  $\Delta^{\otimes 4}$ by philosopher's stone 
  $(P,\frac13 R^* +\lambda \hat Q,R ,\frac13 P^*+\lambda \hat S,T+\lambda \hat T)$
\end{theorem}
\begin{proof}
By computations, we observe that  $\mathfrak L_0 P = R\Delta$,  
 $3^{-1} PR^*   + T \mathfrak L_0  = \Id$, $\Delta \hat  Q  = \hat  S \mathfrak L_0$, $ P \hat Q = - \hat T \mathfrak L_0$ and 
  $R \hat  S = - A \hat T$.
 Then  it is sufficient to combine
 Proposition \ref{prop-dilatation}, Proposition \ref{prop-elem8} and 
 Proposition \ref{prop-gauge} to conclude. 
\end{proof}  
%

In the case where the parameter $\nu$ is not zero, we adapt the definitions of 
 $P$ and $R$  by setting, for any $C^\infty$ functions $(v,w)$  with values in $\R^3 \times \R $, 
\begin{align} 
 \label{defUetP3D}   
P(v,w)  &:=   \Big( (\frac32-2\nu)v -\frac12 x \div v - \frac12  x\wedge (\curl v - \nabla  w)  , -2\div v \Big)  , 
 \\   \label{defR3D}  
  R (v,w) &:= \Big( (\frac52-2\nu)v -\frac12   x \div v  - \frac12  x\wedge (\curl v- \nabla  w)    , -\frac12  x \cdot   v  \Big),
    \end{align}
while we introduce the operators: 
\begin{align} 
   \label{defW3D} 
 Q (u,p) &:=  \Big(  \frac{1}{1- \frac{4 \nu}{3} } \big( \frac23 u - \frac16 (p x - \frac{1}{1-\nu} x \wedge \curl u )\big) , -  \frac16  x \cdot \curl u \Big) ,
  \\    \label{defS3D}  
    S(u,p) &:=  \Big( \frac{1}{(1-\nu)(1- \frac{4 \nu}{3})} \Big( \frac{1-2\nu }{3}  u - \frac16 ( {x   (- \div u + \Delta p)} - {x \wedge \curl u} - 2 \nabla p), 
     \\    \nonumber 
   &\quad \quad  -  \frac16  x \cdot \curl u \Big)  .
\end{align}
Observe that from the expressions above, it is natural to exclude 
two values of $\nu$, again the critical value $1$ (regarding the ellipticity), and also $3/4$ (rather than $1/2$, as it was in the section devoted to the $2$D case). 
Next result establishes in particular that for $\nu \in \R \setminus \{\frac34 ,1 \}$,   the operator  
  $\mathfrak L_\nu$,  defined in  \eqref{Eq:LM-Op}   on $\R^3$, 
   is a  basic   transmutation of  $\Delta^{\otimes 4}$  by  $(P,Q,R,S)$. 
\begin{theorem}\label{th-autre}
 For $\nu \in \R$, the operator  
  $\mathfrak L_\nu$   is a post-catalysis  of $\Delta^{\otimes 4}$ by $(P,R)$.
Moreover, 
for $\nu \in \R \setminus \{\frac34 ,1 \}$,  $\mathfrak L_\nu$ is a pre-catalysis  of $\Delta^{\otimes 4}$ by $(Q,S)$.
\end{theorem}
This result can also be proved by direct computations. 

%
\subsection{Consequences for local properties}
\label{cons3D}

Below we deduce from the previous results some local properties of  the $3$D Stokes operator, as we did in Section \ref{cons2D} for the $2$D Lam\'e-Navier operator. 
\medskip

There and here we did not mention the hypoellipticity result which can be deduced from Theorem \ref{th:transfert-HYPO},
 Theorem \ref{th:main2D}/Theorem \ref{th:main3D2} and from the hypoellipticity of the Laplace operator, since in this case this can be deduced easily from 
 a classical result, due to L. Schwartz, stating that a necessary and sufficient condition for a linear, constant coefficient differential operator to be hypoelliptic   is that the operator has a fundamental solution with singular support consisting of the origin alone. 
\medskip

As a consequence of Theorem \ref{th:transfert-WUC} and Theorem \ref{th:main3D2} we obtain the following result regarding the weak unique continuation. 
\begin{corollary} 
\label{th-WUC-stokes}
For  any  non-empty open subset $\tilde \Omega \subset  \R^3$, 
 for any non-empty open subset $\Omega$  of $\tilde \Omega$,  
  for any function $(u,p)$ in $C^\infty (\tilde \Omega ; \R^3 \times \R)$ satisfying $\Delta u = \nabla p $ and $ \div u= 0$ in  $\tilde \Omega$ and $(u,p)=0$ in $ \Omega$,
   then  $(u,p)=0$ in $\tilde \Omega$.
\end{corollary}
A similar result has already been obtained by Fabre and Lebeau  in  \cite{fabrelebeau} with a completely different approach based on Carleman estimates and pseudo-differential analysis. The viewpoint there is to consider the $3$D Stokes operator as the composition of the Leray-Helmholtz projection, which is a non-local operator, and of the Laplace operator. On the opposite,  here, we link the  $3$D Stokes operator to the Laplace operator by means of five local operators, the ones of the philosopher's stone above, what is arguably a good deal.
\medskip

Let us also mention the result by Dehman and Robbiano in  \cite{DR} regarding the Lam\'e-Navier operator in the elliptic regime.  
However, in the latter,  as already mentioned below \eqref{LN-eq},  there is no non-local feature. 
\medskip

As a consequence of Theorem \ref{th:transfert-SUC} and Theorem \ref{th:main3D2} we recover the following result regarding the weak unique continuation, which was obtained in  \cite{Regbaoui} by refining the technics in \cite{fabrelebeau}.
\begin{corollary} 
\label{th-SUC-stokes}
For any non-empty open subset $\Omega$ of $ \R^3$, 
 for any $x_0$ in $\Omega$, 
  for any $(u,p)$ in $C^\infty (\Omega ; \R^3 \times \R)$
     satisfying $\Delta u = \nabla p $ and $ \div u= 0$ in  $ \Omega$  
  and  that there exists $R > 0$ such that, for any $N \in \N$, 
  there exists $C_N > 0$, for any $r$ in $(0,R)$, 
  \begin{equation*}
 \int_{B(x_0,r)}  (\vert  u \vert^2 + \vert  p \vert^2)  \, dx \leq C_N \, r^N ,
\end{equation*}
then $(u,p)=0$ in $ \Omega$.
\end{corollary}
%


  Finally, as a direct consequence of Theorem \ref{th:transfert}, Theorem \ref{th:main3D2} and Theorem  \ref{th:approxharm1}, we have the following new result, regarding the $3$D steady Stokes operator  $\mathfrak L_0$, defined in  \eqref{Eq:LM-Op} with $\nu =0$.
%
\begin{corollary}\label{th:rw3D}
The $3$D steady Stokes operator  $\mathfrak L_0$  satisfies  the Runge property with point controls on  $\R^3$.
\end{corollary}
 Corollary \ref{th:rw3D} extends the result obtained by Glass and Horsin in \cite{GH}  
  where the authors make use of nonlocal operators and fail to prove the full Runge property in particular regarding the nature of the poles. 
 Such a result, which can be thought  in controllability theory as 
an approximate controllability result with point controls was actually the starting point of our investigations. 
\medskip

Finally as a consequence of Theorem \ref{th-lap-poly}, of Theorem \ref{th:transfert-poly} and Theorem \ref{th:main3D2}
we also have the following result regarding the  polynomial Runge property. 
\begin{corollary}\label{co-p}
The   $3$D steady Stokes operator  $\mathfrak L_0$  satisfies  the  polynomial Runge property.
\end{corollary}

\subsection{Generation of fundamental solutions}

As a corollary of  Theorem \ref{th:main3D2} and Theorem \ref{th-autre}, 
we  have the following  result regarding the 
generation of  fundamental solutions of  the  Lamé-Navier   and Stokes systems by the means of 
fundamental solutions of  the Laplace problem.  

\begin{corollary} \label{th-fonda3D}
Let $(b,c)$  in $\R^3 \times \R$.  
Let $\Omega$  a domain in $\R^3$ with $0 \in \Omega$.
 \begin{enumerate}[leftmargin=19pt]
\item Let $\nu \in \R$.  Let $(v,w):\Omega \rightarrow \R^3\times \R$  satisfies 
$(\Delta  v,\Delta  w) = (b \delta_0,c \delta_0)$. 
Set $ (u,p) :=  P (v,w)$, where $P$ is given by  \eqref{defUetP3D}.
 Then $\mathfrak L_\nu (u,p)=( (2-2\nu) b \delta_0 , 0)$ in $ \Omega$.
 In the particular case where 
 $(v,w):\Omega \rightarrow \R^3\times \R$ is 
 given, for $x$ in $\R^3$ by 
$v(x) := \frac{ 1 }{ | x |}  b$ and $w (x) := \frac{ c }{ | x |}  $,  
 then
 \begin{equation}
  \label{osee}
 u(x) = \frac{ 1 }{ | x |}  \Big(  (1- 2\nu)\Id +  \frac{ x }{ | x |}  \otimes \frac{ x }{ | x |}  \Big) b \quad  \text{ and }  \quad p=   \frac{ 2 b \cdot x}{ | x |^3}  .
 \end{equation}
\item Let  $\nu \in \R \setminus \{\frac34 ,1 \}$ and 
$(u,p):\Omega \rightarrow \R^3\times \R$  
such that $\mathfrak L_\nu (u,p)=( (2-2\nu) b \delta_0 , 0)$ in $ \Omega$. Set $(v,w) := Q(u,p)$, where $Q$ is given by   \eqref{defW3D}. 
Then $ \Delta v   = b \delta_0$ and $w$ is harmonic. 
If moreover  $\nu=0$, then $ (u,p) :=  P (v,w)$.   
\end{enumerate}
\end{corollary}
\begin{proof}
By the first part of Theorem \ref{th-autre}, 
 $\mathfrak L_\nu(u,p)= R( b  \delta_0 ,c \delta_0 ) $.
Some elementary computations  provide:
\begin{equation} \label{elem}
 x \cdot \curl (b  \delta_0) =0, 
\quad x \, \div (b  \delta_0)  = -  b \delta_0 ,
\quad x \wedge \curl (b  \delta_0) =2b  \delta_0 
 \,   \text{ and }  \, 
  x \wedge \nabla  \delta_0 = 0 
\end{equation}
so that  $R ( b  \delta_0 ,c \delta_0 ) = ((2-2\nu)b  \delta_0 ,0)$, and therefore   $\mathfrak L_\nu(u,p)=  ((2-2\nu)b  \delta_0 ,0).$
 The particular case can be proved by direct computation.
For the second part, we first use the second part of Theorem \ref{th-autre},  
$\Delta (v,w) =S  \mathfrak L_\nu (u,p) = S((2-2\nu)b  \delta_0 ,0)$, and using  \eqref{elem} again, we arrive at $\Delta (v,w) = (b \delta_0 ,0)$.
 Finally, in the case where $\nu=0$, we apply the pre-regeneration identity (from  Theorem \ref{th:main3D2})  and observe that $T( b \delta_0 , 0)= 0$ 
 to conclude that $(u,p) = P(v,w)$. 
\end{proof}
In the  peculiar case of Corollary  \ref{th-fonda3D}
where $b=0$, $c=0$ and $\nu=0$, we recover some earlier results on the representation,  in an arbitrary domain, of solutions to the $3$D steady Stokes system   by Zsupp\'an in \cite{Zsuppan}.

\subsection{The Lorentz reflection operator as a proper silvern differential self-transmutation}

In this section we reinterpret the  Lorentz reflection introduced in 
\cite{Lorentz} as a proper silvern differential transmutation of the $3$D Stokes operator of itself by a non-trivial philosopher's stone (by opposition to the one in Proposition \ref{prop-elem0}). 
We also refer to \cite{KK}, where the  Lorentz reflection  at stake here is tagged as the Lorentz ``hat" operator.  This operator is instrumental in the process to extend a solution to the $3$D Stokes system in a half-space, with a Dirichlet boundary condition, 
extending the classical Schwarz principle for the harmonic functions. 
 Lorentz ``hat"  operator should be rather thought as a preparation to reflection for the Stokes operator, while the corresponding ``hat"  operator for the harmonic functions  simply is the identity.

 Let us define the matrix:
  $$J_3 :=  \begin{bmatrix}
 1&0 &   0
  \\ 0  &  1 &   0 
  \\ 0 & 0 & -1
 \end{bmatrix} .
  $$
We also 
  define the two following linear differential operators with variable coefficients: $P_L$ and $Q_L$, defined by their action on $C^\infty$ functions $(v,w)$  with values in $\R^3 \times \R $ by 
\begin{align*} 
P_L(v,w)  &:=   \Big( -J_3 v - 2 x_3 \nabla v_3 +  x_3^2 \nabla w   , w + 2x_3 \partial_3 w - 4 \partial_3 v_3  \Big)  , 
 \\  
R_L   (v,w) &:= \Big( - J_3 v - 2x_3  \nabla  v_3 +  x_3^2 ( \nabla \Delta w - \nabla \div v)  , 
\\ \ &\quad \quad\quad- w - 2 x_3 v_3 + x_3^2 \Delta w - x_3^2 \div v \Big).
    \end{align*}
 \begin{proposition} \label{Lorentz}
  The  $3$D Stokes operator 
 $\mathfrak L_0 $, that is the operator defined in  \eqref{Eq:LM-Op} with $\nu =0$, 
 is a proper silvern differential transmutation of itself by the philosopher's stone 
 $(P_L,P_L,R_L,R_L, 0)$. 
\end{proposition}
\begin{proof}
It is sufficient to check that  $AP_L = R_L A$ and that $P_L P_L = \Id$.
  \end{proof}
   A similar result holds in $2$D. Let us mention that such an operator can be used to obtain  quite easily the fundamental solutions of  the $2$D and $3$D Stokes operators in a half-space.

\section{Appendix}

We recall from  \cite{Axler} 
the following logarithmic conjugation result. 
 \begin{lemma} \label{Axler}
Let $\Omega$ a finitely path-connected domain of $\R^2$ and let us denote by $(H_j)_{1 \leq j \leq r}$ the holes. 
Let $u$ which satisfies $\Delta u =0$ in $\Omega$. For any $1 \leq j \leq r$, let $a_j \in H_j$.
Then there are a holomorphic function $f$ and  some real numbers  $(c_j )_{1 \leq j \leq r}$ such that, in $\Omega$, 
\begin{equation}
  \label{log-f}
u = \Re f + \sum_{1 \leq j \leq r} c_j \log | \cdot - a_j  |.
\end{equation}
\end{lemma}
We refer to \cite{Axler}.  for the proof. Let us only highlight that the holomorphic function $f$ and  the real numbers  $(c_j )_{1 \leq j \leq r}$ are given explicitely. 
Indeed,  for $1 \leq j \leq r$, 
$$c_j := \frac{ 1 }{ 2\pi i}  \int_{\gamma_j } h(w) \, dw,$$
where   $h := u_x -i u_y$,  $\gamma_j $ a smooth curve with winding number $\delta_{ij}$ around the hole $H_j$, 
and $f$ is given 
 for any $z$ in $\Omega$, by
$$f(z) := \int_{\mathcal C(b,z)} ( h(w) - \sum_j \frac{ c_j }{ w-a_j}  )\, dw,   $$
where $b$ in $\Omega$ is arbitrarily fixed, as well as  the smooth  curve $\mathcal C(b,z)$ between $b$ and $z$. 

\begin{proof}[Proof of  Theorem \ref{th:approxharm1} in the case of a finitely path-connected domain of $\R^2$]
Let $u$ which satisfies $\Delta u =0$ in $\Omega$. 
Then, by Lemma \ref{Axler}, there are a holomorphic function $f$ and  some real numbers  $(c_j )_{1 \leq j \leq r}$ such that 
 $u$ satisfies  \eqref{log-f}. 
Then by the Runge theorem, for any $\eps >0$, 
there is a rational function $\overline f$ with poles in the $(a_j )_{1 \leq j \leq r}$ such that 
  $\|f - \overline f  \|_{C^k}  \leq \eps$. 
  Next, set 
  $$\overline u = \Re \overline f + \sum_{1 \leq j \leq r} c_j \log | \cdot - a_j  |. $$
 Then $\Delta \overline u $ is a combination of derivatives of Dirac masses at the $(a_j )_{1 \leq j \leq r}$ and   $\|u - \overline u  \|_{C^k}  \leq  \|f - \overline f  \|_{C^k}  \leq \eps$.
\end{proof}


\bigskip \ \par \noindent {\bf Acknowledgements.} F. S. was partially supported by the Agence Nationale de la Recherche, Project SINGFLOWS, grant ANR-18-CE40-0027-01, 
Project MOSICOF, grant  ANR-21-CE40-0004, Project BOURGEONS, grant ANR-23-CE40-0014-01, and Institut Universitaire de France.

\end{document}